\documentclass[11pt]{article}
\usepackage{amsmath,amssymb,amscd,color}
\newcommand{\sidenote}[1]{\marginpar{{ \footnotesize #1}}}

\newcommand{\ep}{{ \epsilon  }}

\newcommand{\bq}{\begin{equation}}
\newcommand{\eq}{\end{equation}}
\newcommand{\pa}{\partial}
\newcommand{\R}{{ \mathbb{R}  }}

\newcommand{\bbr}{{ \mathbb{R}  }}

\newcommand{\bke}[1]{\left( #1 \right)}
\newcommand{\bkt}[1]{\left[ #1 \right]}

\newcommand{\norm}[1]{\left\Vert #1 \right\Vert}
\newcommand{\abs}[1]{\left| #1 \right|}

\newcommand{\om}{{ \omega  }}

\newcommand{\na}{\nabla}

\newcommand {\al}{\alpha}
\newcommand {\be}{\beta}

\newcommand{\ddt}{\frac{d}{dt}}

\newcommand{\wx}{\langle x \rangle}
\newcommand{\Del}{\Delta}
\newcommand{\sca}{\frac{n}{\phi}}
\newcommand{\psca}{ \left(\sca\right)}
\newcommand{\tsca } {\left (\sca -K \right)_+}
\newcommand{\nt} { (n-K)_+}
\newcommand{\ntt} { (n-\xi\eta(t))_+}
\newcommand{\ct} { (c-K)_+}
\newcommand{\ctt} { (c-\xi\eta(t))_+}

\newcommand{\qed}{\hfill\fbox{}\par\vspace{.2cm}}
\begin{document}
\bibliographystyle{plain}


\newtheorem{defn}{Definition}
\newtheorem{lemma}[defn]{Lemma}
\newtheorem{proposition}{Proposition}
\newtheorem{theorem}[defn]{Theorem}
\newtheorem{cor}{Corollary}
\newtheorem{remark}{Remark}
\numberwithin{equation}{section}

\def\Xint#1{\mathchoice
   {\XXint\displaystyle\textstyle{#1}}%
   {\XXint\textstyle\scriptstyle{#1}}%
   {\XXint\scriptstyle\scriptscriptstyle{#1}}%
   {\XXint\scriptscriptstyle\scriptscriptstyle{#1}}%
   \!\int}
\def\XXint#1#2#3{{\setbox0=\hbox{$#1{#2#3}{\int}$}
     \vcenter{\hbox{$#2#3$}}\kern-.5\wd0}}
\def\ddashint{\Xint=}
\def\dashint{\Xint-}
\def\aint{\Xint\diagup}

\newenvironment{proof}{{\bf Proof.}}{\hfill\fbox{}\par\vspace{.2cm}}
\newenvironment{pfthm1}{{\par\noindent\bf
            Proof of Theorem \ref{Theorem3} }}{\hfill\fbox{}\par\vspace{.2cm}}
\newenvironment{pfthm2}{{\par\noindent\bf
            Proof of Theorem  \ref{Theorem4} }}{\hfill\fbox{}\par\vspace{.2cm}}
\newenvironment{pfthm3}{{\par\noindent\bf
Proof of Theorem \ref{blowup2}} }{\hfill\fbox{}\par\vspace{.2cm}}
\newenvironment{pfthm4}{{\par\noindent\bf
Sketch of proof of Theorem \ref{Theorem6}.
}}{\hfill\fbox{}\par\vspace{.2cm}}
\newenvironment{pfthm5}{{\par\noindent\bf
Proof of Theorem 5. }}{\hfill\fbox{}\par\vspace{.2cm}}
\newenvironment{pflemsregular}{{\par\noindent\bf
            Proof of Lemma \ref{sregular}. }}{\hfill\fbox{}\par\vspace{.2cm}}

\title{Global existence and temporal decay in Keller-Segel models coupled to fluid equations}
\author{Myeongju Chae, Kyungkeun Kang and Jihoon Lee}

\date{}

\maketitle
\begin{abstract}
  We consider a Keller-Segel model coupled to the
  incompressible Navier-Stokes equations in spatial dimensions two and three.
  We establish the local existence of regular solutions and
  present some blow-up criteria for both cases that equations of oxygen concentration
  is of parabolic or hyperbolic type. We also prove global existence and decay
  estimate in time under the some smallness conditions of initial
  data.
\end{abstract}

\section{Introduction}
 \setcounter{equation}{0}
 In this paper, we consider mathematical models  describing the
 dynamics of  oxygen, swimming bacteria, and
 viscous incompressible fluids in
$\bbr^d$, with $d=2,\,3$. Bacteria or microorganisms often live in
fluid, in which the biology of chemotaxis is intimately related to
the surrounding physics.
Such a model was proposed by Tuval et al.\cite{TCDWKG} to describe
the dynamics of swimming bacteria, {\it Bacillus subtilis}. We
consider the following equations in \cite{TCDWKG} and set $Q_{T} =
(0,\, T] \times \R^{d}$ with $d=2,\, 3$:
 \begin{equation}\label{KSNS} \left\{
\begin{array}{ll}
\partial_t n + u \cdot \nabla  n - \Delta n= -\nabla\cdot (\chi (c) n \nabla c),\\
\vspace{-3mm}\\
\partial_t c + u \cdot \nabla c-\mu\Delta c =-k(c) n,\\
\vspace{-3mm}\\
\partial_t u + u\cdot \nabla u -\Delta u +\nabla p=-n \nabla
\phi,\quad
\nabla \cdot u=0 
\end{array}
\right. \quad\mbox{ in }\,\, (x,t)\in \R^d\times (0,\, T],
\end{equation}
where $c(t,\,x) : Q_{T} \rightarrow \R^{+}$, $n(t,\,x) : Q_{T}
\rightarrow \R^{+}$, $u(t,\, x) : Q_{T} \rightarrow \R^{d}$ and
$p(t,x) :  Q_{T} \rightarrow \R$ denote the oxygen concentration,
cell concentration, fluid velocity, and scalar pressure,
respectively. The nonnegative function $k(c)$ denotes the oxygen
consumption rate, and the nonnegative function $\chi (c)$ denotes
chemotactic sensitivity. Initial data are given by $(n_0(x), c_0(x),
u_0(x))$. We study both cases that either $\mu=1$ or $\mu=0$ in the
equation of oxygen and, for convenience, the case $\mu=1$ of
\eqref{KSNS} is called parabolic Keller-Segel-Navier-Stokes
equations (abbreviated to P-KSNS) and the case $\mu=0$ is referred
as partially parabolic-hyperbolic Keller-Segel-Navier-Stokes
equations (abbreviated to PH-KSNS). We will refer the system to the
Keller-Segel-Stokes equations (abbreviated to PH-KSS) if  the
convection term $u\cdot \na u$ is absent in
$(1.1)_3$.

To describe the fluid motions, we use Boussinesq approximation to
denote the effect due to heavy bacteria. The time-independent
function $\phi =\phi (x)$ denotes the potential function produced by
different physical mechanisms, e.g., the gravitational force or
centrifugal force. Thus, $\phi(x)=ax_{d}$ is one example of gravity
force, and $\phi(x)=\phi(|x|)\rightarrow 0$ as $|x| \rightarrow \infty$ is an example of  centrifugal force.
\\
\indent The classical model to describe the motion
of cells was suggested by Patlak\cite{Patlak} and
Keller-Segel\cite{KS1, KS2}. It consists of a system of the dynamics
of  cell density $n=n(t,x)$ and the concentration of chemical
attractant substance $c=c(t,x)$ and is given as
\begin{equation}\label{KS-nD} \,\,\left\{
 \begin{array}{c}
 n_t=\Delta n-\nabla \cdot(n \chi\nabla c),\\
 \vspace{-3mm}\\
 \alpha c_t=\Delta c-\tau c+n,\\
 \end{array}
 \right.
\end{equation}
where $\chi$ is the sensitivity and $\tau^{-\frac12}$ represents the
activation length. The system \eqref{KS-nD} has been extensively
studied by many authors and we will not try to give list of results
here (see e.g. \cite{Her-Vela, Horst-Wang, NSY, OY, Win} and
references therein).
\\ \indent
Our main objective of this paper is to present blow-up criteria of
\eqref{KSNS} in two or three dimensions, unless solutions exist
globally in time (see Theorem \ref{Theorem3} and Theorem
\ref{blowup2} below), and to establish global existence of regular
solutions and their decay properties, when certain norm of initial
data is sufficiently small (see Theorem \ref{Theorem4} and Theorem
\ref{Theorem6} below).

We mention previously known results related to ours. In \cite{Lorz}
local existence of solutions was shown in three dimensional bounded
domains and \cite{DLM} proved the global-in-time existence of the
smooth solutions when initial data are close to constant states in
$\R^3$ and $\chi(\cdot), k(\cdot)$ satisfy certain conditions. More
precisely, \cite{DLM} showed that if initial data
$\|(n_0-n_{\infty},\, c_0,\, u_0)\|_{H^3}$ is sufficiently small,
then there exists a unique global solution, provided that
\begin{equation}\label{Assumption1}
\chi'(\cdot)\ge 0,\quad k'(\cdot) >0,\quad
\left(\frac{k(\cdot)}{\chi(\cdot)} \right)^{''} <0.
\end{equation}
In the absence of the fluid in \eqref{KSNS}, i.e., $u=0$,
\cite{YTao} showed that there exists a unique, global and bounded
solution if $\chi$ is sufficiently small, dependent upon
$\norm{c_0}_{L^{\infty}(\R^d)}$.





For two dimensional case, in \cite{Liu-Lorz}, Liu and Lorz showed
the global existence of a weak solution in $\R^2$ under the following conditions on $\chi(\cdot)$ and $k(\cdot)$:
\begin{equation}\label{Assumption1-2}
\chi'(\cdot)\ge 0,\quad (\chi(\cdot)k(\cdot))' >0,\quad
\left(\frac{k(\cdot)}{\chi(\cdot)} \right)^{''} <0.
\end{equation}

In two dimensions, Winkler\cite{Win2} proved the
global existence of regular solutions without smallness assumptions
on initial data for bounded domains with boundary conditions
$\partial_{\nu}n=\partial_{\nu} c=u=0$ under the following sign
conditions on $\chi(\cdot)$ and $k(\cdot)$:
\begin{equation}\label{Assumption2}
\left( \frac{k(\cdot)}{\chi(\cdot)} \right)^{'} >0, \quad
(\chi(\cdot) k(\cdot))' \ge 0, \quad \left(
\frac{k(\cdot)}{\chi(\cdot)} \right) ^{''} \le 0.
\end{equation}
%
%
In \cite{ckl} the authors of the paper established global existence
of smooth solutions in $\R^2$ 
with no
smallness of the initial data and a certain conditions, motivated by
experimental results in \cite{CFKLM} and \cite{TCDWKG}, on
$\chi(\cdot)$ and $k(\cdot)$ (compare to \eqref{Assumption2}), that
is,
\begin{equation} \label{CKL10-march14}
\chi(c),\, k(c),\, \chi'(c),\, k'(c)\geq 0,\,\mbox {and }\, \sup |
\chi (c)- \mu k(c)| < \epsilon\,\,\mbox{ for some }\,\mu>0.
\end{equation}
Construction of weak solutions in $\R^3$ was also discussed in
\cite{ckl} with replacement of $\abs{\chi (c)- \mu k(c)}=0$ in
\eqref{CKL10-march14}. We refer to \cite{FLM}, \cite{Tao-W} and
\cite{CKK} and references therein for the nonlinear diffusion models
of a porous medium type $\Delta n^m$, instead of $\Delta n$.

As mentioned earlier, our main motivation is to study existence of
regular solutions of \eqref{KSNS} when certain norm of initial data
is small. To be more precise, we show that in case $\mu=1$, if
$\norm{c_0}_{L^{\infty}}$ is small, then solutions become regular in
$\R^d$, $d=2,3$ and satisfy a certain degree of decay in time (when
$d=3$, Stokes system is under our consideration for fluid
equations). On the other hand, in case $\mu=0$, we establish local
solutions in time, and then global solutions with time decay if
$\norm{n_0}_{L^{\frac{d}{2}}}$ is small. We first consider the case
(P-KSF) in \eqref{KSNS}. Local-in-time existence of classical
solutions for \eqref{KSNS} was established in Theorem $1$ in
\cite{ckl} such that $(n,c,u) \in L^{\infty}(0,T;
H^{m-1}(\bbr^d)\times H^m(\bbr^d) \times H^m(\bbr^d))$ for some
$T>0$ and $m\ge 3$. We present some blow-up criteria of local
classical solutions, unless the maximal existence time is infinite.

\begin{theorem}\label{Theorem3}
Let the initial data $(n_0, c_0, u_0)$ be given in $H^{m-1}(\bbr^d)\times H^m (\bbr^d)\times H^m(\bbr^d)$
for $m\geq 3$ and $d=2,3$. Assume that $\chi, k, \chi', k'$ are all non-negative and
 $\chi$, $k\in C^m(\R^+)$ and $k(0)=0$, $\|
\nabla^l \phi \|_{L^{\infty}}<\infty$ for $1\le |l|\le m$. If the
maximal time of existence, $T^*$, in Theorem~$1$ in \cite{ckl}, is
finite, then one of the following  is true in each case of $\R^2$ or
$\R^3$, respectively:
\begin{equation}\label{CK10L-march13}
(2D)\qquad
\norm{n}_{L^q(0,T^*;L^{p}(\R^2))}=\infty,\qquad\frac{2}{p}+\frac{2}{q}=2,\,\,1<l\leq
\infty.
\end{equation}
\begin{equation}\label{CK20L-march13}
(3D)\qquad \norm{u}_{L^{\gamma}(0,T^*;L^{\beta}(\R^3))}+
\norm{n}_{L^q(0,T^*;L^{p}(\R^3))}=\infty,
\end{equation}
where
\[
\frac{3}{\beta}+\frac{2}{\gamma} \leq 1,\,\,3<\beta\leq \infty,\quad
\frac{3}{p}+\frac{2}{q}=2,\,\,\frac{3}{2}<p\leq \infty.
\]
If fluid equation is the Stokes system for $(3D)$,
$\norm{u}_{L^q(0,T^*;L^{p}(\R^3))}$ in \eqref{CK20L-march13} is
dropped.
\end{theorem}
The proof of Theorem \ref{Theorem3} will be given in section 2.
\begin{remark}
We remind the following scaling invariance of \eqref{KSNS}:
\begin{equation}\label{C10KL-march15}
n_R(t,x) : = R^2 n(R^2t, Rx), \quad c_R(t,x) = c(R^2t, Rx), \quad u_R
(t,x) = R u(R^2t, Rx)
\end{equation}
and observe that \eqref{CK10L-march13} and \eqref{CK20L-march13} are
invariant functionals under the scaling \eqref{C10KL-march15}. For
the limiting case $(l,m)=(1,\infty)$ in \eqref{CK10L-march13}, due
to conservation of total mass, we note that
$\norm{n}_{L^{\infty}(0,t;L^{1}(\R^2))}=\norm{n_0}_{L^1(\R^2)}<\infty$
for any $t<T^*$. In Proposition \ref{l1small} in section 2, we prove
that if $\norm{n_0}_{L^1(\R^2)}$ is sufficiently small, blow-up does
not occur in a finite time. We, however, leave an open question
whether or not singularity may develop for large $L^1$ norm of
$n_0$.
\end{remark}
\begin{remark}
Liu and Lorz\cite{Liu-Lorz} showed global-in-time existence of weak solution to \eqref{KSNS} in two dimensional case under the assumption \eqref{Assumption1-2}. Since their weak solution satisfies integrability $n \in L^2(0, T; L^2(\R^2))$ for any $T>0$, which is a special case in \eqref{CK10L-march13}, their weak solution is, in fact, a classical solution if $(n_0, c_0, u_0)\in H^{m-1}(\bbr^d)\times H^m (\bbr^d)\times H^m(\bbr^d)$ for $m\geq 3$ as a consequence of Theorem \ref{Theorem3}.
\end{remark}

The second result is the existence of regular solutions under the
assumption that $\norm{c_0}_{L^{\infty}}$ is sufficiently small.
\begin{theorem}\label{Theorem4}
Let the assumptions in Theorem \ref{Theorem3} hold. We consider the
Navier-Stokes equations in $\R^2$ and the Stokes system in $\R^3$ in
$\eqref{KSNS}_3$. There exists a constant $ \delta>0 $ such that if
$\norm{c_0}_{L^{\infty}(\R^d)}<\delta$, then classical solution of
\eqref{KSNS} exists globally. Furthermore, $n$ and $c$ satisfy the
following time decay:
\begin{equation}\label{CK30L-march14}
\| n(t)\|_{L^{\infty}(\R^d)} + \|c(t)\|_{L^{\infty}(\R^d)} \le
C(1+t)^{-\frac d4},\qquad d=2, 3.
\end{equation}
\end{theorem}
The proof of Theorem \ref{Theorem4} will be given in section 3. Next
we study (PP-KSF), namely
\begin{align}\label{adKS}
\begin{aligned}
\begin{cases}
\pa_t n +  u\cdot \na n - \Delta n = - \na\cdot [\chi(c)n\na c] \\
\pa_t c +u\cdot\na n =- k(c)n\\
\partial_t u + u\cdot \nabla u -\Delta u +\nabla p=-n \nabla
\phi,\quad
\nabla \cdot u=0.
\end{cases}
\end{aligned}\quad t>0,\,\, x\in \bbr^d\end{align}
When the fluid is absent, the Keller-Segel equations with chemical
of ODE type, typically referred to the angiogenesis system, has been
studied in \cite{CPZ1,CPZ2,CPZ3} and \cite{PeVa}:
\begin{equation}\label{degKS}
\pa_t n  - \Delta n = - \na\cdot [\chi(c)n\na c],\quad \pa_t c  =-
c^mn\qquad t>0,\,\, x\in \bbr^d.
\end{equation}
%

In section 4, we show local classical solution of \eqref{adKS} by
the usual iteration method and present blow-up criteria of
\eqref{adKS}, if a finite time singularity occurs. Now we state the
third main result.
\begin{theorem}\label{blowup2}
Let $m\ge 3$ and $d=2,3$. Assume that $\chi(\cdot), k (\cdot) \in
C^{m}(\bbr)$, $\norm{\nabla^l \phi}_{L^{\infty}}<\infty$ for
$1\le\abs{l}\le m$. Then there exists $T^*$, the maximal existence
time, such that if $(n_0, c_0, u_0) \in H^{m}(\bbr^d)\times
H^{m+1}(\bbr^d) \times H^m(\bbr^d)$, then there exists a unique
classical solution of \eqref{adKS} satisfying for any $t<T^*$
\[
(n,c,u) \in C(0,t; H^{m}(\bbr^d)\times H^{m+1}(\bbr^d) \times
H^m(\bbr^d)),
\]
\[
(n,u) \in L^2 (0,t; H^{m+1}(\bbr^d) \times H^{m+1}).
\]
Furthermore,
if $T^{*}<\infty$, then one of the following  is true in each case
of $\R^2$ or $\R^3$, respectively:
\begin{equation}\label{CKL10-april3}
(2D)\qquad
\norm{n}_{L^2(0,T^*;L^{\infty}(\R^2))}=\infty
\end{equation}
\begin{equation}\label{CKL20-april3}
(3D)\qquad\norm{u}_{L^{\gamma}(0,T^*;L^{\beta}(\R^3))}+\norm{n}_{L^2(0,T^*;L^{\infty}(\R^3))}=\infty,
\qquad \frac{3}{\beta}+\frac{2}{\gamma} \leq 1,\,\,3<\beta\leq
\infty.
\end{equation}
If fluid equation is the Stokes system for $(3D)$,
$\norm{u}_{L^q(0,T^*;L^{p}(\R^3))}$ in \eqref{CKL20-april3} is
dropped.
\end{theorem}

Last main result is global existence of regular solutions for
\eqref{adKS} and their decays in time, when $\norm{n_0}_{L^{\frac
d2}(\R^d)}$ is sufficiently small. To be more precise, we obtain the
following:
\begin{theorem}\label{Theorem6}
Let $d=2, 3$ and we consider the Navier-Stokes equations in $\R^2$
and the Stokes system in $\R^3$ in $\eqref{adKS}_3$. Suppose that
$(n_0,c_0,u_0) \in H^{m}(\bbr^d)\times H^{m+1}(\bbr^d) \times
H^m(\bbr^d))$ for $m\ge 3$. Then there exists
$\epsilon_2=\epsilon_2(d, \| c_0\|_{L^{\infty}})$ such that if
$\norm{n_0}_{L^{\frac d2}(R^d)}<\epsilon_2$, then solutions of
\eqref{adKS} become global and classical. Furthermore, $n$ satisfies
the following time decay:
\begin{equation}\label{CK30L-april6}
\norm{n(t)} _{L^{\infty}} \le  C(1+t)^{-1} .
\end{equation}
\end{theorem}

This paper is organized as follows. In Section 2, we prove Theorem
\ref{Theorem3} by obtaining a priori estimates. Section 3 is devoted
to prove Theorem \ref{Theorem4} by adjusting De Giorgi method
introduced in \cite{PeVa}. We obtain the blow-up criteria in Theorem
3 by using a priori energy estimates in Section 4 and The proof of
Theorem 4 is presented again by using De Giorgi method in the last
section.
%

\section{Blow-up criteria of parabolic system}

We first recall following blow-up criteria for \eqref{KSNS} obtained
in \cite[Theorem 2]{ckl}:
\begin{equation}\label{2d-reg}
(2D)\qquad \int_0^{T^*}\norm{\nabla c}^2_{L^{\infty}(\R^2)}=\infty,
\end{equation}
\begin{equation}\label{3d-nse-reg}
(3D)\qquad
\int_0^{T^*}\norm{u}^{\gamma}_{L^{\beta}(\R^3)}+\int_0^{T^*}\norm{\nabla
c}^2_{L^{\infty}(\R^3)}=\infty, \qquad
\frac{3}{\beta}+\frac{2}{\gamma}=1,\,\,3<\beta\leq \infty,
\end{equation}
where $T^*$ is the maximal time of existence. If the fluid equation
is the Stokes system, not the Navier-Stokes equations, in case of
$3D$, then the condition on $v$ in \eqref{3d-nse-reg} is not
necessary and thus it can be dropped. From now on, we denote
$L^{q,p}_{t,x}=L^q(0, T^{*}; L^p(\R^d))$ and $L^p_{t,x}=L^p(0,
T^{*}; L^p(\R^d))$, unless any confusion is to be expected. All
generic constants will be written by $C$, which may change from one
line to the other and $\epsilon$ will be used to indicate some
sufficiently small positive number.
\\
\begin{pfthm1} We argue by contradiction.
We suppose that \eqref{CK10L-march13} in 2D and
\eqref{CK20L-march13} in 3D are finite. We then show that $T^*$
cannot be a finite maximal time of existence, which will lead to a
contradiction. We start with the case of dimension two.
\\
$\bullet$\,\,({\bf{2D case}})\quad We first suppose
$\|n\|_{L^2_{t,x}([0,T^{*})\times \bbr^2)}$ is finite. We will show
$\int_0^{T^{*}} \| \nabla c\|_{L^{\infty}}^2 dt<\infty$, which is
contrary to the blow-up criterion \eqref{2d-reg} proved in
\cite{ckl}. In the following, we obtain a priori estimates, since
our computations are made for any time $T$ with $T<T^*$.
We frequently use the following type of interpolation
inequality
\bq\label{inter}
\| D^l f \|_{L^p(\R^2)} \leq C\|
f\|_{L^q(\R^2)}^{\theta} \| D^k f \|_{L^r(\R^2)}^{1-\theta},
\eq
where $0\le l \leq k-1$ and
$l-\frac{2}{p}=-\theta\frac{2}{q}+(1-\theta)(k-\frac{2}{r})$ with
$0\le\theta\le 1$. From maximum principle for $c$ and conservation
of mass for $n$, it is immediate that $\|
c\|_{L^{\infty}_{t,x}}\le\|c_0 \|_{L^{\infty}}$ and $\| n
\|_{L^{\infty,1}_{t,x}} \leq \| n_0 \|_{L^1}$. We note that the
convection term, $(u\cdot\nabla) c$, is estimated as follows:
\[
\| (u\cdot \nabla) c\|_{L^2_{t,x}}^2
\leq C \int_0^T \| u \|_{L^2} \| \nabla u \|_{L^2} \|
c\|_{L^{\infty}} \| \Delta c \|_{L^2} dt
\]
\begin{equation}\label{C10KL-march26}
\leq \epsilon \| \Delta c \|_{L^2_{t,x}}^2+C_{\epsilon} \| u
\|_{L^{\infty, 2}_{t,x}}^2 \|c\|_{L^{\infty}_{t,x}}^2\| \nabla
u\|_{L^2_{t,x}}^2.
\end{equation}
Via $L^2$-estimate of the heat equation, we have
\begin{equation}\label{C20KL-march26}
\| c_t \|_{L^2_{t,x}}^2+\| \Delta c \|_{L^2_{t,x}}^2 \leq C\|
c_0\|_{H^1}^2+C \| n \|_{L^2_{t,x}}^2+ \| (u\cdot \nabla)
c\|_{L^2_{t,x}}^2.
\end{equation}
Combining \eqref{C10KL-march26} and \eqref{C20KL-march26} together
with the hypothesis $\| n \|_{L^2_{t,x}}^2<\infty$, we obtain $\|
\Delta c \|_{L^2_{t,x}} < \infty$.
%
%
%
On the other hands, $L^2$ scalar product for equation of $n$ gives
that
\[
\frac{d}{dt} \| n \|_{L^2}^2 +\| \nabla n \|_{L^2}^2 \leq C \| n \nabla c \|_{L^2} \| \nabla n \|_{L^2}
\leq \frac14 \| \nabla n \|_{L^2}^2 +C\| n \|_{L^4}^2 \| \nabla c\|_{L^4}^2
\]
\[
\leq \frac14 \| \nabla n \|_{L^2}^2 +C\| n \|_{L^2} \| \nabla n \|_{L^2} \| c\|_{L^{\infty}} \| \Delta c \|_{L^2}
\leq \frac12 \| \nabla n \|_{L^2}^2 +C\| c\|_{L^{\infty}}^2 \| \Delta c\|_{L^2}^2 \| n \|_{L^2}^2.
\]
Using Gronwall's inequality, we have
\[
 \| n \|_{L^{\infty,2}_{t,x}}^2 +\| \nabla n \|_{L^2_{t,x}}^2
 \leq \| n_0\|_{L^2}^2 \exp\left( C\|c\|_{L^{\infty}_{t,x}}^2\| \Delta c \|_{L^2_{t,x}} ^2  \right).
\]
Next, testing $-\Delta c$ with equation of $c$, we obtain
\[
\frac12\frac{d}{dt} \| \nabla c\|_{L^2}^2 +\| \Delta c \|_{L^2}^2
\leq \int_{\R^2} |\nabla u ||\nabla c|^2 dx+C \| n \|_{L^2}^2 +\frac14 \| \Delta c \|_{L^2}^2
\]
\[
\leq C \| \nabla u \|_{L^2} \| \nabla c\|_{L^4}^2+C \| n \|_{L^2}^2 +\frac14 \| \Delta c \|_{L^2}^2
\leq C \| \nabla u \|_{L^2}^2\| \nabla c\|_{L^2}^2+C \| n \|_{L^2}^2 +\frac12 \| \Delta c \|_{L^2}^2.
\]
Again, Gronwall's inequality impiles that
\[
\| \nabla c \|_{L^{\infty,2}_{t,x}}^2 +\| \Delta c \|_{L^2_{t,x}}^2 <\infty.
\]
The energy estimate of vorticity equation leads to
\[
\frac12\frac{d}{dt}\| \omega\|_{L^2}^2 +\| \nabla \omega\|_{L^2}^2
\leq C\| n\|_{L^2}^2 + \frac14 \| \nabla \omega \|_{L^2}^2.
\]
Therefore, we obtain $\| \omega\|_{L^{\infty,2}_{t,x}}^2+\| \nabla
\omega \|_{L^{2,2}_{t,x}}^2<\infty$ via Gronwall's inequality.
Finally, we have
\[
\frac12\frac{d}{dt} \| \Delta c\|_{L^2}^2+\| \nabla \Delta c\|_{L^2}^2 \leq C \| \nabla u \nabla c\|_{L^2}^2 +C\| u \nabla^2 c\|_{L^2}^2+ C\| \nabla c n \|_{L^2}^2 +C \| \nabla n \|_{L^2}^2
\]
\[
\leq C\| \nabla u \|_{L^3}^2 \| \nabla c \|_{L^6}^2+C\| u \|_{L^{\infty}}^2 \| \Delta c \|_{L^2}^2 +C\| \nabla c \|_{L^6}^2 \| n \|_{L^3}^2 +C\| \nabla n \|_{L^2}^2.
\]
This gives the bound of $\| \nabla \Delta c \|_{L^{2,2}_{t,x}}^2$ and in turn,
 the bound of $\| \nabla c \|_{L^{2, \infty}_{t,x}}$.
We complete the proof for the case $\| n \|_{L^2_{t,x}}<\infty$.\\
\indent In case that $\| n \|_{L^{p,q}_{x,t}}< \infty$ with
$\frac{2}{p}+\frac{2}{q}\leq 2$ and $p>2$,  we observe that $\| n
\|_{L^{2}_{x,t}}< \infty$. Indeed, due to interpolation and
H\"older's inequality,
\[
\int_0^T \| n \|_{L^2}^2 dt \leq C\int_0^T \| n \|_{L^1}^{\frac{p-2}{p-1}}
\| n \|_{L^p}^{\frac{p}{p-1}} dt \leq C\| n \|_{L^{p,q}_{x,t}}^{\frac{p}{p-1}}.
\]
Hence, with the aid of previous result of $L^2_{x,t}$ case, the case
$\frac{2}{p}+\frac{2}{q}\leq 2$ with $p>2$ is direct. It remains to
consider the case $\| n \|_{L^{p,q}_{x,t}}< \infty$ with
$\frac{2}{p}+\frac{2}{q}= 2$ and $1<p<2$. We note first that
\[
\frac12\frac{d}{dt} \| u \|_{L^2}^2+\| \nabla u \|_{L^2}^2 \leq C\| n \|_{L^p} \| u \|_{L^{\frac{p}{p-1}}}
\]
\[
\leq C\| n \|_{L^p} \| u \|_{L^2}^{\frac{2(p-1)}{p}} \| \nabla u \|_{L^2}^{\frac{2-p}{p}}
\leq C \| n \|_{L^p}^{\frac{2p}{3p-2}} \| u \|_{L^2}^{\frac{4(p-1)}{3p-2}} +\frac12 \| \nabla u \|_{L^2}^2.
\]
Therefore, due to Gronwall's inequality, $\sup \| u \|_{L^2}^2 +\int_0^T \| \nabla u \|_{L^2}^2 dt< \infty.$\\
Using the mixed norm estimate of the heat equation for $c$, we have
\begin{equation}\label{C30KL-march26}
\| c_t\|_{L^{p,q}}^q+\| \Delta c \|_{L^{p,q}}^q \leq C\| c_0
\|_{H^2}^q+ C\| u\cdot \nabla c \|_{L^{p,q}}^q +C \| n
\|_{L^{p,q}}^q.
\end{equation}
Noting that $\frac{(p-1)q}{p}= 1$, we compute

\[
\int_0^T \| u \cdot \nabla c \|_{L^p}^q dt \leq \int_0^T \| u \|_{L^{2p}}^q \| \nabla c\|_{L^{2p}}^q dt
\leq C\int_0^T \| u \|_{L^{2p}}^q \| c\|_{L^{\infty}}^{\frac{q}{2}} \| \Delta c\|_{L^p}^{\frac{q}{2}} dt
\]
\begin{equation}\label{C40KL-march26}
\leq C\int_0^T \| u \|_{L^{2p}}^{2q} dt+ \epsilon \int_0^T \| \Delta
c\|_{L^p}^q dt
\leq C \| u \|_{L^{2,\infty}_{x,t}}^{\frac{2q}{p}}\int_0^T \| \nabla
u \|_{L^2}^{\frac{2(p-1)q}{p}} dt +\epsilon \int_0^T \| \Delta
c\|_{L^p}^q dt.
\end{equation}
Combining \eqref{C30KL-march26} and \eqref{C40KL-march26} with
sufficiently small $\epsilon>0$, we have
\[
\| c_t\|_{L^{p,q}}^q+\| \Delta c \|_{L^{p,q}}^q \leq C.
\]
Multiplying the equation of $n$ with $\ln n$ and integrating it by
parts, we obtain
\[
\frac{d}{dt} \int n \ln n dx +\| \nabla \sqrt{n}\|_{L^2}^2
\leq C \| \nabla c\|_{L^{\frac{2p}{2-p}}} \| \nabla  \sqrt{n}\|_{L^2} \| \sqrt{n} \|_{L^{\frac{p}{p-1}}}
\]
\begin{equation}\label{C50KL-march26}
\leq C\| \Delta c\|_{L^p}\| \sqrt{n}\|_{L^2}^{\frac{2(p-1)}{p}} \|
\nabla \sqrt{n} \|_{L^2}^{\frac{2}{p}}
\leq C \| \Delta c\|_{L^{p}}^q\| \sqrt{n}\|_{L^2}^2 +\frac12 \| \nabla \sqrt{n}\|_{L^2}^2.
\end{equation}
Using Gronwall inequality, the estimate \eqref{C50KL-march26} leads
to $\nabla \sqrt{n}\in L^{2}_{t,x}$, which implies $\| n
\|_{L^{2}_{t,x}}<\infty$. This completes the proof for 2D case.\\
\\
$\bullet$\,\,({\bf{3D case}})\quad Suppose that
\eqref{CK20L-march13} is not true. As in 2D case, we then show
$\int_0^{T^*} \| \na c\|_{L^{\infty}(\R^3)}< \infty$, which is
contrary to the blow-up criterion \eqref{3d-nse-reg} proved in
\cite{ckl}.  The proof of the case for Stokes system is omitted,
since its verification is simpler.
%
We first show $\nabla c\in L^{\infty}(0,T^*; L^2)$ and $\nabla^2
c\in L^2(0,T^*; L^2)$. Testing $\log n$ to the equation
$\eqref{KSNS}_1$,
\[
\frac{d}{dt}\int n\log n +4\int\abs{\nabla n^{\frac{1}{2}}}^2 =\int
\chi(c)\nabla n\nabla c\leq C\int\abs{\nabla
n^{\frac{1}{2}}}n^{\frac{1}{2}}\abs{\nabla c}
\]
\begin{equation}\label{CKL10-march12}
\leq C\norm{\nabla
n^{\frac{1}{2}}}_{L^2}\norm{n^{\frac{1}{2}}}_{L^{2p}}\norm{\nabla
c}_{L^{\frac{2p}{p-1}}}\leq \epsilon\norm{\nabla
n^{\frac{1}{2}}}^2_{L^2}+C_{\epsilon}\norm{n}_{L^{p}}\norm{\nabla
c}^2_{L^{\frac{2p}{p-1}}}.
\end{equation}
Via Gagliardo-Nirenberg's inequality, we note
$\norm{\nabla c}_{L^{\frac{2p}{p-1}}}\leq C\norm{\nabla
c}^{\frac{2p-3}{2p}}_{L^2}\norm{\nabla^2 c}^{\frac{3}{2p}}_{L^2}$,
and therefore, combining \eqref{CKL10-march12}, we obtain
\begin{equation}\label{CKL30-march12}
\frac{d}{dt}\int n\log n +\int\abs{\nabla n^{\frac{1}{2}}}^2 \leq
\epsilon\norm{\nabla^2
c}^2_{L^2}+C_{\epsilon}\norm{n}^{\frac{2p}{2p-3}}_{L^{p}}\norm{\nabla
c}^2_{L^{2}}.
\end{equation}
Multiplying $\eqref{KSNS}_1$ with $-\Delta c$ and using integration
by parts,
\begin{equation}
\frac{1}{2}\frac{d}{dt}\int \abs{\nabla c}^2 +\int\abs{\Delta c}^2
=\int u\nabla c\Delta c+\int k(c)n\Delta c:=I+J.
\end{equation}
We first consider the term $J$. Following the same computations in
\eqref{CKL10-march12}-\eqref{CKL30-march12},
\begin{equation}\label{CKL40-march12}
J=-\int k'(c)n\abs{\nabla c}^2-\int k \nabla n \nabla c\leq\ep \|
n^{\frac 12} \|_{L^2}^2+ \epsilon\norm{\nabla^2
c}^2_{L^2}+C_{\epsilon}\norm{n}^{\frac{2p}{2p-3}}_{L^{p}}\norm{\nabla
c}^2_{L^{2}},
\end{equation}
where we used that $k'\ge 0$. On the other hand, $I$ is
estimated as follows:
\[
I\leq \norm{u}_{L^{\beta}}\norm{\nabla
c}_{L^{\frac{2\beta}{\beta-2}}}\norm{\Delta c}_{L^2}\leq
C\norm{u}_{L^{\beta}}\norm{\nabla
c}^{\frac{\beta-3}{\beta}}_{L^2}\norm{\nabla^2
c}^{\frac{\beta+3}{\beta}}_{L^2}
\]
\begin{equation}\label{CKL50-march12}
\leq \epsilon \norm{\nabla^2 c}^2_{L^2} +
C_{\epsilon}\norm{u}^{\frac{2\beta}{\beta-3}}_{L^{\beta}}\norm{\nabla
c}^2_{L^2}
\end{equation}
Summing up \eqref{CKL30-march12}, \eqref{CKL40-march12} and
\eqref{CKL50-march12},
\begin{equation}\label{CKL60-march12}
\frac{d}{dt}(\int n\log n +\frac{1}{2}\norm{\nabla c}^2_{L^2})+
\norm{\nabla n^{\frac{1}{2}}}^2_{L^2} +\norm{\nabla^2 c}^2_{L^2}
\leq C_{\epsilon}(\norm{n}^{\frac{2p}{2p-3}}_{L^{p}}
+\norm{u}^{\frac{2\beta}{\beta-3}}_{L^{\beta}})\norm{\nabla
c}^2_{L^2}.
\end{equation}
Due to Gronwall inequality, we observe that $\nabla c\in
L^{\infty}(0,T^*; L^2)$ and $\nabla^2 c\in L^2(0,T^*; L^2)$.

Next we consider the vorticity equation of fluid equations in
\eqref{KSNS}
\[
\partial_t \omega -\Delta \omega + u\cdot \nabla \omega-\omega \cdot \nabla u
=-\nabla\times (n \nabla \phi).
\]
Energy estimate shows
\begin{equation}\label{CKL65-march12}
\frac{1}{2}\frac{d}{dt}\norm{\omega}^2_{L^2}+\norm{\nabla\omega}^2_{L^2}=\int\omega
\cdot \nabla u\omega -\int\nabla\times (n \nabla
\phi)\omega:=K_1+K_2.
\end{equation}
First, we estimate $K_2$. Following computations as in above,
\[
K_2\leq C\int\abs{\nabla n}\abs{\omega}\leq C\norm{\nabla
n^{\frac{1}{2}}}_{L^2}\norm{n^{\frac{1}{2}}}_{L^{2p}}\norm{\omega}_{L^{\frac{2p}{p-1}}}
\leq C\norm{\nabla
n^{\frac{1}{2}}}_{L^2}\norm{n}^{\frac{1}{2}}_{L^{p}}\norm{\omega}^{\frac{2p-3}{2p}}_{L^2}
\norm{\nabla\omega}^{\frac{3}{2p}}_{L^2}
\]
\begin{equation}\label{CKL70-march12}
\leq \epsilon\norm{\nabla
n^{\frac{1}{2}}}^2_{L^2}+\epsilon\norm{\nabla\omega}^2_{L^2}
+C_{\epsilon}\norm{n}^{\frac{2p}{2p-3}}_{L^{p}}\norm{\omega}^2_{L^{2}},
\end{equation}
On the other hand, for $K_1$, similarly as in \eqref{CKL50-march12},
we show
\begin{equation}\label{CKL80-march12}
K_1\leq \epsilon \norm{\nabla\omega}^2_{L^2} +
C_{\epsilon}\norm{u}^{\frac{2\beta}{\beta-3}}_{L^{\beta}}\norm{\omega}^2_{L^2}.
\end{equation}
Adding above estimate together and using Gronwall inequality, we
observe that $\omega\in L^{\infty}(0,T^*; L^2)$ and $\nabla\omega
\in L^2(0,T^*; L^2)$.

Next, considering four cases of $\frac 32 < p<2$, $ 2\le p \le 3$, $
3< p \le 6$, and $p>6$ separately,  we will show that
\begin{equation}\label{udot}
\|  u\nabla c \| _{ L^q(0,T^*;L^p)}\le C + \ep \| \na ^2 c\| _{
L^q(0,T^*;L^p)}.
\end{equation}
The proof of \eqref{udot} will be given later. Then by the maximal
regularity of heat equation,
\begin{equation}\label{CKL10-april6}
 \| c_t\| _{ L^q(0,T^*;L^p)} + \|\Del c\| _{ L^q(0,T^*;L^p)} \le C
  + \ep \|\na^2 c\| _{ L^q(0,T^*;L^p)} + C \| n\| _{ L^q(0,T^*;L^p)},
\end{equation}
we obtain $ \| \na ^2 c\| _{ L^q(0,T^*;L^p)} <\infty$. Now we turn
to show that $n \in L^{\infty}(0, T^*; L^r)$ for any $r>1$. Testing
$n^{r-1}$ to $\eqref{KSNS}_1$ and noting $\chi'\ge0$, we observe
that
\[
\frac{1}{r}\frac{d}{dt}\norm{n}_{L^r}^r+\frac{4(r-1)}{r^2}\norm{\nabla
n^{\frac{r}{2}}}^2_{L^2}=-\frac{r-1}{r}\bke{\int\chi' \abs{\nabla
c}^2n^r+\int\chi \Delta c n^r}\leq C\int\abs{\Delta c} n^r
\]
\[
\leq C\norm{\Delta c}_{L^p}\norm{n^r}_{L^{\frac{p}{p-1}}}\leq C
\norm{\Delta
c}_{L^p}\norm{n^r}^{\frac{2p-3}{2p}}_{L^{1}}\norm{n^r}^{\frac{3}{2p}}_{L^{3}}
\]
\begin{equation}\label{CKL100-march12}
\leq C \norm{\Delta
c}_{L^p}\norm{n^r}^{\frac{2p-3}{2p}}_{L^{1}}\norm{\nabla
n^{\frac{r}{2}}}^{\frac{3}{l}}_{L^{2}} \leq C_{\epsilon}\norm{\Delta
c}^{\frac{2p}{2p-3}}_{L^p}\norm{n}^r_{L^r}+\epsilon\norm{\nabla
n^{\frac{r}{2}}}^2_{L^2}.
\end{equation}
Since $\nabla^2 c\in L^q(0,T^*;L^p)$, via Gronwall inequality, we
can prove that $n\in L^{\infty}(0,T^*;L^r)$ for any $r>1$. Let us
choose $r>3$ and via \eqref{CKL10-april6} we then obtain that $\|\na
c\|_{L_t^2 L_x^{\infty}} <\infty$. Indeed, we note that $ \| u\na
c\|_{L^r} < C_{\ep} + \ep \| \na^2 c\|_{L^r}$ for $3< r \le 6 $ (see
\eqref{3to6} below). Again by the  maximal regularity of heat
equation, we have
\[\| c_t\|_{L^2(0,T^*; L^r)} + \| \Del c\|_{L^2(0,T^*; L^r)}
\le C_{\ep} + \ep \| \na^2 c\|_{L^2(0,T^*; L^r)}.\]
Combined with $\| \na c\|_{L^{\infty}_tL^2_x} + \| \na^2 c\|_{L^2_{t,x}} <\infty$,
 the above yields to  $\| \na c\|_{L^2_tL^{\infty}_x} <\infty$ as
 desired,
which is contrary to a blow-up criterion \eqref{3d-nse-reg}.
\\
\indent It remains to show the estimate \eqref{udot}. By   $ (\na c,
\om) \in {L^{\infty}_tL^2_x}$ and $ (\na^2 c, \na \om)\in
{L^2_{t,x}}$, it follows that $u\nabla c$ belongs to
$L^4(0,T^*;L^2)\cap L^2(0,T^*;L^3)$ and so $u\nabla c\in
L^q(0,T^*;L^p)$ for $2\le p\le 3$. The other cases are treated as follows.\\
\noindent  (i) (Case $ \frac 32 < p <2$)\quad  Setting
$p^*=3p/(3-p)$, we have
\[
\norm{u\nabla c}_{L^p}\leq \norm{u}_{L^6}\norm{\nabla
c}_{L^{\frac{6p}{6-p}}}\leq C\norm{\omega}_{L^2}\norm{\nabla
c}^{\frac{p}{5p-6}}_{L^2}\norm{\nabla
c}^{\frac{4p-6}{5p-6}}_{L^{p^*}}\le C_\ep + \ep \| \na^2 c\|_{L^p},
\]
where we use that $2<6p/(6-p)<p^*$. Taking $L^q$-norm in time
variable,  it follows that $u\nabla c\in
L^q(0,T^*;L^p)$.\\
\noindent (ii)  (Case $3< p \le 6$)\quad Using that $\omega$ and
$\nabla c$ are in $L^{\infty}(0,T^*; L^2)$, we have
\[
\norm{u\nabla c}_{L^p}\leq
\norm{u}^{\frac{6-p}{2p}}_{L^2}\norm{u}^{\frac{3(p-2)}{2p}}_{L^6}\norm{\nabla
c}_{L^{\infty}}\leq
C\norm{u}^{\frac{6-p}{2p}}_{L^2}\norm{\omega}^{\frac{3(p-2)}{2p}}_{L^2}
\norm{\nabla c}^{\frac{2(p-3)}{5p-6}}_{L^2}\norm{\nabla^2
c}^{\frac{3p}{5p-6}}_{L^p}
\]
\begin{equation*}
\leq C\norm{u}^{\frac{6-p}{2p}}_{L^2}\norm{\nabla^2
c}^{\frac{3p}{5p-6}}_{L^p}\leq
C_{\epsilon}\norm{u}^{\frac{5p-6}{2(p-3)}}_{L^2}+\epsilon\norm{\nabla^2
c}_{L^p}.
\end{equation*}
On the other hands, due to $p/(p-1)\le q=2p/(2p-3)$, we note that
\begin{equation}\label{C20KL-march12}
\norm{n}^2_{L^2}\leq \norm{n}^{\frac{p-2}{p-1}}_{L^1}
\norm{n}^{\frac{p}{p-1}}_{L^p}\leq C+C\norm{n}^q_{L^p}.
\end{equation}
From $(1.1)_3$, we have
\begin{equation}\label{C60KL-march26}
\ddt \|u\|_{L^2}^2 +\|\na u\|_{L^2}^2 \le C\|u\|_{L^2}^2 +
\|n\|_{L^2}^2.
\end{equation}
Combining \eqref{C20KL-march12} and \eqref{C60KL-march26}, we obtain
$\|u\|_{L^{\infty}_t L^2_x} + \| \na u\|_{L^2_{t,x}} < \infty$.
Therefore we obtain
\begin{equation}\label{3to6}
\norm{u\nabla c}_{ L^p}\leq C_\ep+\epsilon\norm{\nabla^2 c}_{
L^p}\qquad \mbox{ for any }\,\,t<T^*.
\end{equation}
\noindent (iii) (Case $ p >6$)\quad We estimate
\[
\norm{u\nabla c}_{L^p}\leq\norm{u}_{L^p}\norm{\nabla
c}_{L^{\infty}}\leq
C\norm{u}^{\frac{p-2}{p}}_{L^{\infty}}\norm{u}^{\frac{2}{p}}_{L^2}\norm{\nabla
c}^{\frac{2(p-3)}{5p-6}}_{L^2}\norm{\nabla^2
c}^{\frac{3p}{5p-6}}_{L^p}\leq
C\norm{u}^{\frac{p-2}{p}}_{L^{\infty}}\norm{\nabla^2
c}^{\frac{3p}{5p-6}}_{L^p}
\]
\[
\leq C\norm{\nabla u}^{\frac {p-2}{2p}}_{L^{2}}\norm{\nabla^2
u}^{\frac {p-2}{2p}} \norm{\nabla^2 c}^{\frac{3p}{5p-6}}_{L^p} \leq
C\norm{\nabla \om}^{\frac{p-2}{2p}}_{L^{2}}\norm{\nabla^2
c}^{\frac{3p}{5p-6}}_{L^p}.
\]
Therefore, suing $q=2p/(2p-3)$ and Young's inequality, we have
\begin{equation}\label{C300KL-march13}
\norm{u\nabla c}^q_{L^p}
\leq C_{\epsilon}\norm{\nabla \omega}^2_{L^2}+\epsilon\norm{\nabla^2
c}^{\frac{12p^2}{(5p-6)(3p-4)}}_{L^p} \leq
C_{\epsilon}(1+\norm{\nabla \omega}^2_{L^2})+\epsilon\norm{\nabla^2
c}^{q}_{L^p},
\end{equation}
where we used that $\frac{12p^2}{(5p-6)(3p-4)}+ <q=\frac{2p}{2p-3}$
with $p>6$. Therefore, the estimate \eqref{udot} is also true for
$p
>6$.  This completes the proof.
\end{pfthm1}

Next we present the proof of existence of regular solutions in case
$\norm{n_0}_{L^1}$ is small in dimension two.
\begin{proposition}\label{l1small}
Let $d=2$ and initial data, $\chi$, $k$, and $\phi$ satisfy the
assumptions in Theorem \ref{Theorem3}. Assume further that $\| n_0
\ln n_0\|_{L^1(\R^2)} + \| \wx n_0\|_{L^1(\R^2)}$ is finite. Then,
there exists an $\ep>0$ such that if $\|n_0\|_{L^1} <\ep$ then the
maximal time of existence, $T^*$, is infinite, i.e. $T^*=\infty$.
\end{proposition}
\begin{proof}
Estimates in this proof are a priori, since all computations are
made before the maximal time of existence, $T^*$. We note that, due
to the conservation of mass, $ \displaystyle \sup_{0\le t < T^*} \|
n\|_{L^1} < \ep$ and therefore, we have
\begin{equation}\label{l2small}
\|n\|_{L^2}\le C\| \sqrt n\|_{L^2} \| \na \sqrt n\|_{L^2} \le C\ep
\| \na \sqrt  n\|_{L^2}.
\end{equation}
Multiplying equations of $n, c$ in \eqref{KSNS} with $\ln n$,
$\Delta c$, respectively, we obtain
\begin{equation}\label{lnn}
\ddt\int n \ln n dx+ \int |\na \sqrt n|^2 dx   = - \int \chi(c)
\Delta c n dx \le C \ep \| \Delta c\|_{L^2} \| \na  \sqrt n\|_{L^2}.
\end{equation}
\[
\frac 12 \ddt \| \na c\|_{L^2}^2 +\| \Delta c\|_{L^2}^2 \le C \ep \|
\nabla\sqrt n\|_{L^2}\| \Delta c\|_{L^2} + C \| \na u\|_{L^4} \| \na
c\|_{L^4} \| \na c\|_{L^2}
\]
\begin{equation}\label{delc}
\le C\ep \| \nabla\sqrt n\|_{L^2}^2+ \frac 12\| \Delta c\|_{L^2}^2 +
C \| \na u\|_{L^2}\| \na c \|_{L^2}^3 + \frac 14 \| \nabla
\omega\|_{L^2}^2.
\end{equation}
Adding  \eqref{lnn} and \eqref{delc} with the following estimate:
\[
\frac 12 \ddt\| \omega\|_{L^2}^2 + \| \na \omega\|_{L^2}^2
\le C \| n\|_{L^2}^2 + \frac 14 \| \na \omega\|_{L^2}^2,
\]
we have
\[
\ddt \bke{  \|\na c\|_{L^2}^2 + \| \omega\|_{L^2}^2 + \int n \ln n
dx } + \| \na \sqrt n\|_{L^2}^2 + \|\Del c\|_{L^2}^2 +\| \na \omega
\|_{L^2}^2
\]
\[
\le C \| \na c\|_{L^2}^2 ( \| \na u \|_{L^2}^2 + \|\na c\|_{L^2}^2).
\]
By Gronwall's inequality, we obtain
\begin{align} \label{indefinite}\begin{aligned}
& \bke{  \|\na c\|_{L^2}^2 + \| \omega\|_{L^2}^2 + \int n \ln n dx }
+ \int_0^t \bke{\| \na \sqrt n\|_{L^2}^2 + \|\Del c\|_{L^2}^2 +\| \na \omega \|_{L^2}^2 } \, ds\\
&\quad \le  \bke{ \|\na c_0\|_{L^2}^2 + \| \omega_0\|_{L^2}^2 + \int
_0 n_0 \ln n_0 dx} \exp \left ( C \| c_0 \|^2_{L^2} \right).
\end{aligned}\end{align}
where we used that $\norm{ \na c }_{L^2_{x,t}}\leq
\norm{c_0}_{L^2}$. Next, we estimate $\int n \abs{\ln n} dx$. For
simplicity, we set
\[
D_1=\{x: n(x)\leq e^{-\abs{x}}\},\qquad D_2=\{x:
e^{-\abs{x}}<n(x)\leq 1\}.
\]
A typical argument for dealing with kinetic entropy (see e.g.
\cite{DiLi}), we estimate
\[
\left|\int n(\ln n)_ - \right|=-\int_{D_1} n\ln n-\int_{D_2} n\ln
n\leq C\int_{D_1}\sqrt{n}+\int_{D_2} \wx n \leq C\int
e^{-\frac{\abs{x}}{2}}+\int \wx n,
\]
where $(\ln x)_-$ is a negative part of $\ln x$ and
$\wx=(1+\abs{x}^2)^{\frac{1}{2}}$. We compute
\begin{equation}\label{xn}
\ddt\int_{\R^2} \wx n dx = \int_{\R^2} nu \na \wx dx + \int_{\R^2} n
\Del \wx dx + \int_{\R^2} \chi(c) n \na c \na \wx dx.
\end{equation}
The term   $\int_{\R^2} nu \na \wx dx $ is estimated by
\[
\abs{\int_{\R^2} nu \na \wx dx }\le  \ep \|\na \sqrt
n\|_{L^2}\|u\|_{L^2}.
\]
Noting that $\abs{\na \wx}+\abs{\Del \wx}\leq C$, we get
\[
\abs{\int_{\R^2} n \Del \wx dx} + \abs{\int_{\R^2} \chi(c) n \na c
\na \wx dx}
\leq C\ep(1 +\norm{\nabla \sqrt{n}}_{L^2}\norm{\nabla c}_{L^2}).
\]
Thus, \eqref{xn} is estimated as follows:
\begin{align}\label{equ33}
\ddt \int_{\R^2} \wx n dx \le C\ep( \|\na \sqrt n \|_{L^2}^2 
+ \|u\|_{L^2}^2 +  \norm{\nabla c}^2_{L^2} +1).
\end{align}
From the $u$-equation we have
\[\ddt \| u\|_{L^2}^2 +\|\na u\|_{L^2}^2 \le C\ep \| \na \sqrt n\|_{L^2} \|u\|_{L^2},\]
which gives
\begin{equation}\label{equ34} \| u\|_{L^2}^2 + \int_0^t \|\na u\|_{L^2}^2 ds \le \|u_0\|_{L^2} +
 C\ep \int_0^t \| \na \sqrt n\|^2 + \| u\|_{L^2}^2 ds.
 \end{equation}
We add $ 2\abs{ \int n(\ln n)_- dx }$ to  \eqref{indefinite} and
using \eqref{equ33}, \eqref{equ34} we then have
\[
\bke{  \|\na c\|_{L^2}^2 + \| \omega\|_{L^2}^2 + \int n |\ln n |dx +
\| u\|_{L^2}^2 }
\]
\[
+\int_0^t \bke{\frac 12\| \na \sqrt n\|_{L^2}^2 + \|\Del c\|_{L^2}^2
+\| \na \omega \|_{L^2}^2+ \|\na u\|_{L^2}^2} \, ds
\]
\[
\le  \bke{ \|\na c_0\|_{L^2}^2 + \| \omega_0\|_{L^2}^2 + \int _0 n_0
\ln n_0 dx} \exp \left ( C \| c_0 \|^2_{L^2} \right)
\]
\[
+ C + \int \wx n_0 dx + C\ep \| c_0\|_{L^2}^2 +  C\epsilon \| u_0
\|_{L^2}^2 t.
\]
Therefore, we have $\int_0^T \|\na \sqrt n \|^2_{L^2} dt \le C(T)$,
which implies $n\in L^{2}_{x,t}$ via \eqref{l2small}. Therefore, it
is direct, due to \eqref{CK10L-march13} in Theorem \ref{Theorem3},
that solutions become regular. This completes the proof.
\end{proof}


\begin{section}{Proof of Theorem \ref{Theorem4}}
In this section we present the proof of Theorem \ref{Theorem4}. We
start with the control of $L^p-$norm of $n$ under the smallness of
$\norm{c_0}_{L^\infty}$ in next proposition. We use the similar
weighted energy estimate in \cite{YTao}, which treated the case of
$\chi, k$ are constants and fluid equation is absent. We remark that
due to the incompressible condition of $u$, the proof of \cite[Lemma
3.1]{YTao} can be applicable to our case and the generalization to
non-constant $\chi(c),\kappa(c)$ is also available as long as a
maximum principle of $c$ holds, i.e. $0\le c\le
\|c_0\|_{L^{\infty}}$.


\begin{proposition}\label{alemma}
Let the assumptions in Theorem \ref{Theorem3} hold and $ p \in (1,
\infty)$. There exists $\delta_1=\delta_1(p)$ such that if
$\norm{c_0}_{L^{\infty}}<\delta_1$, then $n(t)\in L^p(\R^d)$ for all
$t\in [0,T^*)$ and
\begin{equation}\label{CKL10-march27}
\| n(t)\|_{L^p} \le C=C( p, \|c_0\|_{L^{\infty}}, \|n_0\|_{L^p}),
\end{equation}
where $T^*$ is the maximal time of existence in Theorem
\ref{Theorem3}.
%
\end{proposition}
\begin{proof}
For a positive $\phi(c)$ such that $\phi'(c) \ge 0$, which will be
determined later, we obtain
\[
\frac 1 p\ddt \int n^p\phi=
 \int n^{p-1}\phi \left( -u\cdot \na n + \Del n - \na\cdot(n\chi \na c)\right)
 + \frac 1 p \int n^p\phi'\left( -u\cdot \na c + \Del c - n k(c) \right)
\]
\[
= \int n^{p-1} \phi\Del n - \int n^{p-1} \phi \chi \na\cdot (n \na
c) +\frac 1p \int n^p \phi' \Del c
\]
\begin{equation}\label{basic1}
- \int n^{p-1} \phi \chi' n |\na c|^2  - \frac 1 p \int n^p \phi' n
k(c) - \int n^{p-1}\phi u\cdot \na n - \frac 1p \int n^p \phi' u
\cdot \na c.
\end{equation}
We note that, due to $ \na\cdot u =0$ such that $\phi' u \cdot \na c
=\na \cdot (\phi u)$, the last two terms in \eqref{basic1} are
cancelled, i.e. $\int n^{p-1}\phi u\cdot \na n +\frac 1p \int n^p
\phi' u \cdot \na c =0$.
%
%
%
Via the integration by parts, we have
\[
\frac 1 p\ddt \int n^p\phi  +  (p-1) \int n^{p-2} \phi |\na n|^2 +
\frac 1p \int n^p \phi^{''} |\na c|^2
\]
\begin{equation}\label{CKL15-march27}
= -2 \int n^{p-1}\phi' \na n \cdot \na c+ (p-1)\int n^{p-1}\chi \phi
\na n \cdot \na c + \int n^p \chi \phi' |\na c|^2 - \frac{1}{p}\int
n^p \phi' n k.
\end{equation}
Noting that the last term in \eqref{CKL15-march27} is non-positive
and using Cauchy-Schwartz inequality,
%
\begin{align}\label{basic2}
&\frac 1 p\ddt \int n^p\phi  +  \frac{p-1}{2} \int n^{p-2} \phi |\na n|^2 + \frac 1p \int n^p \phi^{''} |\na c|^2\\
\nonumber
& \le \frac{4}{p-1} \int n^{p}\frac{(\phi')^2}{\phi} | \na c|^2+ (p-1)\int n^p\chi ^2\phi |\na c|^2
 + \int n^p\chi \phi' |\na c|^2.
\end{align}
We set $\phi(c) = e^{(\be c)^2}$ and we look for $\phi$ satisfying
\begin{equation}\label{CKL20-march27}
\frac{4}{p-1} \frac{(\phi')^2}{\phi}
+(p-1)\chi ^2\phi
+\chi \phi' \le  \frac {1}{2p}\phi^{''}.
\end{equation}
%
Let $\displaystyle\chi_1 = \sup_{0\le c\le \|c_0\|_{L^{\infty}}}
\chi(c).$ We then see that \eqref{CKL20-march27} is satisfied,
provided that
\begin{align} \label{1-3}
(p-1)\chi_1^2 \le \frac{1}{3p} \be^2, \qquad
\|c_0\|_{L^{\infty}}\chi_1 \le \frac{1}{6p}, \qquad
\frac{8}{p-1}\be^2 \|c_0\|^2_{L^{\infty}} \le \frac{1}{6p}.
\end{align}
If $\be$ is chosen such that $ 6p(p-1)\chi_1^2 = \be^2$ and if
$\chi_1 \|c_0\|_{L^{\infty}} \le \frac{1}{24p}$, it is
straightforward that \eqref{1-3} is satisfied.
%
Therefore, if $\|c_0\|_{L^{\infty}}$ is sufficiently small, we
obtain
\[
\frac 1 p \ddt \int n^p \phi + \frac{p-1}{2} \int n^{p-2}\phi |\na n|^2
+ \frac{1}{2p} \int n^p \phi'' |\na c|^2 \le 0.
\]
Since $\phi >1$, it follows that $ \int_{\R^d} n^p(t)dx \le e^{\be^2
\| c_0\|_{L^{\infty}}^2} \int_{\R^d} n_0^p dx. $ This completes the
proof.
\end{proof}

We remark that Proposition \ref{alemma} dose not give control of
$\|n\|_{L^{\infty}}$. Toward the boundedness as well as the decay of
$\|n\|_{L^{\infty}}$, we modify the approach done in \cite{PeVa},
where the degenerate Keler-Segel system \eqref{degKS} was
considered. With the aid of incompressibility of the velocity vector
field $u$, it turns out that the method of proof in  \cite{PeVa} can
be adjusted properly to our case. Now we are ready to present the
proof of Theorem \ref{Theorem4}.
\\
\begin{pfthm2}
We first note that solutions are classical, because of Theorem
\ref{Theorem3} and Proposition \ref{alemma}. To obtain a truncated
energy inequality for (P-KSNS), we first differentiate $\int \nt^p
\phi $ in time variable, where $\phi$ is the function introduced in
the proof of Proposition \ref{alemma}. Similarly as in
\eqref{basic1}, \eqref{basic2}, we have
\begin{align*}
&\frac 1 p\ddt \int \nt^p\phi  +  (p-1) \int \nt^{p-2} \phi
|\na\nt|^2
+ \frac 1p \int \nt^p \phi^{''} |\na c|^2\\
& =-2 \int \nt^{p-1}\phi' \na \nt \cdot \na c - \underbrace{\frac{1}{p}\int \nt^p \phi' n \kappa}_{ \ge 0}\\
& + (p-1)\int \nt^{p-1}\chi \phi \na \nt \cdot \na c
+ \int \nt^p \chi \phi' |\na c|^2 \\
& +  K(p-1) \int \nt^{p-2} \chi \phi \na\nt \cdot \na c +
K\int \nt ^{p-1} \chi \phi' |\na c|^2 .
\end{align*}
We note that the last two integrands in the equality above are
bounded as follows:
\[
\nt^{p-2} \phi \chi \na \nt \cdot \na c \le
\frac{1}{4K} \nt^{p-2} \phi |\na \nt|^2 + 4K\nt^{p-2} \phi \chi^2 |\na c|^2,
\]
\[
\nt ^{p-1}\chi\phi' |\na c|^2 \le \bke{\frac{1}{8K}\nt^p  + 8K}
\chi\phi' |\na c|^2.
\]
In what follows, we fix $p=2$.
%
%
%
With the aid of \eqref{1-3}, we have
\begin{align}\label{CKL10-march28}
 & \frac 1 2 \ddt \int \nt^2 \phi + \frac 14 \int \phi |\na \nt|^2
+ \frac{1}{8} \int \nt^2  \phi'' |\na c|^2 \nonumber\\
& \le 4K^2 \int \phi\chi^2 |\na c|^2 + 8K^2\int \chi \phi' |\na c|^2
\leq 8K^2\int (\phi\chi^2+\chi \phi') |\na c|^2.
\end{align}
We set $\displaystyle L:=\sup _{ 0\le c \le \| c_0\|_{L^{\infty}}}
\left( \phi\chi^2 + \chi \phi' \right)$. Multiplying the equation
$c$ with $16K^2L$, we have
\begin{align}\label{trivialc}
\ddt \int 8K^2L \ct^2 + \int 16K^2L| \na c |^2 \le 0.
\end{align}
Summing up \eqref{CKL10-march28} and \eqref{trivialc}, we have
\begin{align}\label{ee1}\begin{aligned}
&\ddt \left( \int \frac 1 2 \nt^2\phi  + \int 8K^2L \ct^2 \right)
+ \frac{1}{4} \int  \phi |\na \nt|^2 \\
& + \frac 18 \int \nt^2 \phi^{''} |\na c|^2
+ 8K^2L \int | \na c |^2 \le  0.
\end{aligned} \end{align}
Similarly proceeding as in \cite{PeVa}, we define
\[
U(\xi) = \int_0^T \nu(t) \int  \ntt^2 + \int_0^T \nu(t) \int \ctt^2
:= U_1(\xi)+ U_2(\xi),\quad \xi>0.
\]
Here the auxiliary functions  $\nu(t), \eta(t)$, and the range of
$\xi \in [0, 2M]$ are specified later, where $\nu(t)$ and $\eta(t)$
are decreasing in $t$ and $M$ is a fixed number. We then note that
\begin{align}\label{ufrime}
U'(\xi) = -2\int_0^T \nu(t) \eta(t) \int  \ntt -2 \int_0^T \nu(t)
\eta(t) \int \ctt .
\end{align}
Let $\displaystyle K_1:=\sup_{0\le \xi \le 2M, t>0} \xi \eta(t)$.
Repeating similar computations as in \eqref{ee1}, we observe that
\begin{align}\label{ee2}
\begin{aligned}
&\ddt \left( \int \frac 1 2 \ntt^2\phi  + 8K_1^2L\int \ctt^2 \right)
+ \frac{1}{4} \int  \phi |\na \ntt|^2 \\
& + \frac 18 \int \ntt^2 \phi^{''} |\na c|^2
+ 8K_1^2L \int | \na \ctt |^2 \\
& \le - \xi \eta'(t) \int \ntt\phi dx - \xi\eta'(t) \int \ctt.
\end{aligned} \end{align}
We set $\phi_{\rm{min}}=\min\phi(c)\geq 1$. For simplicity, we
define
\begin{align*}
E(\xi):=& \sup_{0\le t\le T} \left( \int \frac 1 2 \ntt^2
 + \frac{8K_1^2L}{\phi_{\rm{min}}}\int \frac 12 \ctt^2   \right)+
  \frac{1}{4} \int_0^T\int  |\na \ntt|^2\\
 &+ \frac 18 \frac{1}{\phi_{\rm{min}}}\int_0^T \int \ntt^2 \phi^{''} |\na c|^2
 + \frac{8K_1^2L}{\phi_{\rm{min}}}\int_0^T \int | \na \ctt |^2.
 \end{align*}
Then, after integrating \eqref{ee2} in time variable for  $ \xi \ge
\xi_0:=\eta^{-1}(0)\max\{ \| n_0\|_{L^{\infty}}, \|
c_0\|_{L^{\infty}}\}$, we obtain
 \begin{align}\label{energy}
 E(\xi) &\le - \frac{1}{\phi_{\rm{min}}}
 \left( \int_0^T\xi \eta'(t) \int \ntt \phi dx + \int _0^T\xi\eta'(t) \int \ctt \right).
\end{align}
Assuming that $|\eta'(t)| \le C \nu(t) \eta(t)$, which will be
confirmed later, we have
\begin{equation}\label{EU}
E(\xi) \le C\xi |U'(\xi)|.
\end{equation}
We use the Sobolev embedding and then by  interpolating
 we have
 \begin{align}\label{2q}
 \| \ntt \|_{L_{t,x}^q}^2 + \| \ctt\|_{ L_{t,x}^q}^2  \le C E(\xi), \qquad q = 2 (d+2)/d, \quad d\ge 2.
 \end{align}
For simplicity, we denote
\[
A:=\int_0^T \int \nu(t) ^{\frac {1}{\al} } \ntt,\qquad B:=\int_0^T
\int \nu(t) ^{\frac {1}{\al} } \ctt.
\]
Interpolating  $1<2<q$ in space and using the H\"{o}lder's
inequality in time, we have
\begin{equation}\label{CKL20-march28}
U_1(\xi) \le A^{\al}E^{\theta}(\xi),\qquad U_2(\xi)  \le
B^{\al}E^{\theta}(\xi),\qquad \al = \frac {4}{4+d}, \,\, \theta=
\frac{d+2}{d+4}.
\end{equation}
%
%
Under $\nu(t)\eta(t)\ge C \nu(t)^{\frac{1}{\al}}= C\nu(t)^{1+
\frac{d}{4}}$,
we obtain from \eqref{ufrime} together with \eqref{EU} and
\eqref{CKL20-march28}
\begin{equation}\label{CKL30-march28}
| U'(\xi)| \ge CA+CB\ge C U^{\frac {1}{\al}} E^{- \theta/ \al}\ge C
\xi^{-\theta/\al} U ^{\frac{1}{\al}}|U'|^{-\theta/\al},
\end{equation}
where we used that $(A + B)^{\al} \ge C_{\al}(A^{\al} + B^{\al})$
for $0<\al <1$.
%
%
Now we  choose the auxillary functions $\nu$ and $\eta$ by
\[
\nu(t) = (1+t)^{-1}, \qquad \eta(t)= (1+t)^{-\frac{d}{4}}.
\]
By the similar reasoning mentioned in \cite{PeVa}, we need to show
that $U(\xi)$ is finite for some $\xi>0$. Indeed, for $q= {
{2(d+2)}/{d}}$ we have
\begin{align*}
U(2\xi)=  & \int_0^T \int_{\{n \ge  2\xi \eta(t)\}}
\nu(t)\left (  n - 2\xi \eta(t) \right)_+^2 dxdt +  \int_0^T \int_{\{n \ge  2\xi \eta(t)\}}
\nu(t)\left (  c - 2\xi \eta(t) \right)_+^2 dxdt\\
& \le \int _0^T \nu(t) \int_{\{n \ge 2\xi \eta(t)\}} \left ( n  -
\xi \eta(t) \right)_+^{ \frac {2(d+2)}{d}}(\xi \eta(t))^{-\frac 4d}
dxdt\\
&\quad +\int _0^T \nu(t) \int_{\{n \ge 2\xi \eta(t)\}} \left ( c  -
\xi \eta(t) \right)_+^{ \frac {2(d+2)}{d}}(\xi \eta(t))^{-\frac 4d}
dxdt,
\end{align*}
where we used that $ n  - \xi \eta(t) \ge \xi\eta(t)$ on $\{ x| n -
2\xi \eta(t) \ge 0\}$.
Therefore, we obtain
\begin{align}\label{somefi}\begin{aligned}
 U(2\xi) & \le  \int _0^T \frac{\nu(t)}{(\xi\eta(t))^{\frac {4}{d}} }\int
   n ^{ \frac {2(d+2)}{d}} dxdt
  +   \int _0^T \frac{\nu(t)}{(\xi\eta(t))^{\frac {4}{d}} }\int
 c^{ \frac {2(d+2)}{d}} dxdt\\
 &
\le C \xi^{-\frac {4}{d}} \left ( \int_0^T \| n\|_{ \frac {2(d+2)}{d}}^{ \frac {2(d+2)}{d}} dt
+  \int_0^T \|c\|_{ \frac {2(d+2)}{d}}^{ \frac {2(d+2)}{d}} dt \right)\\
& \le C\xi^{-\frac {4}{d}} (  \|n_0\|^{ \frac {2(d+2)}{d}} _2 + \| c_0\|^{ \frac {2(d+2)}{d}}_2),
\end{aligned}\end{align}
where the last inequality in \eqref{somefi} is due to the case of
$K=0$ in \eqref{ee1} together with \eqref{trivialc} and \eqref{2q}.
Therefore, \eqref{somefi} implies that $U(\xi)$ is finite for every
$\xi>0$ as long as $U(\xi)$ exists. Via \eqref{CKL30-march28} and
\eqref{somefi}, we observe that
\[
U'(\xi)\le -C \xi^{-\frac{d+2}{d+6}}U^{\frac{d+4}{d+6}},\qquad \xi
>\xi_0
:= \eta^{-1}(0)\{\|n_0\|_{L^{\infty}}, \|c_0\|_{L^{\infty}}\}.
\]
and it is immediate that $U(\xi)$ vanish at a finite value $\xi =
M(\xi_0, \|n_0\|_{L^2}, \|c_0\|_{L^2}) $.
%
%
Summing up the arguments, we conclude that
$n(x,t) + c(x,t)  \le C M(1+t)^{-\frac d4}$  for $t> 0$.
This completes the proof.
\end{pfthm2}

\begin{remark}
The $L^2$ energy inequality \eqref{ee1} is responsible for the time
decay rate $t^{-d/4}$. The number coincides to that for the solution
of the heat equation with the initial data in $L^2(\bbr^d)$. We do
not know whether or not such decay estimate can be improved, and
thus we leave it an open question.
\end{remark}


\end{section}
\section{Blow up criteria of parabolic-hyperbolic system}

In this section, we consider \eqref{adKS}, which is the case that
equation of $c$ is of no diffusion.
First we construct solutions of \eqref{adKS} locally in time in the
following class of functions:
\[
X_{T}^{s} :=( C([0, T)\,;\, H^{s})\cap L^2(0, T; H^{s+1})\times
C([0, T)\, ; \, H^{s+1})\times ( C([0, T)\,;\, H^{s})\cap L^2(0, T;
H^{s+1}).
\]
Our construction of regular solutions is based on the method of
contraction mapping via linearizing the equations in an iterative
way. Next proposition is the first part of Theorem \ref{blowup2}.
\begin{proposition}\label{construct-sol}
Let initial data, $\chi$, $k$, and $\phi$ satisfy the assumptions in
Theorem \ref{blowup2}. Then there exists $T>0$ depending on $\|
n_0\|_{H^s}, \| c_0 \|_{H^{s+1}}, \| u_0 \|_{H^s}$ with integer
$s>2$ such that a unique solution $(n, c, u) $ in $X_{T}^{s}$
exists.
\end{proposition}
\begin{proof}
 We consider following linearized system, which is
defined iteratively (set $(n^0, c^0, u^0)=(n_0, c_0, u_0)$) over
$\R^d\times (0, T)$ with $d=2, 3$.
\begin{equation}\label{eq1}
\left\{
\begin{array}{l}
 \partial_t n^{(m+1)} + (u^{(m)} \cdot \nabla ) n^{(m+1)} -\Delta n^{(m+1)}=-\nabla \cdot [\chi(c^{(m)}) n^{(m)} \nabla c^{(m)}],\\
 \partial_t c^{(m+1)} +(u^{(m)} \cdot \nabla) c^{(m+1)} =-k (c^{(m)}) n^{(m)},\\
 \partial_t u^{(m+1)}+(u^{(m)} \cdot \nabla )u^{(m+1)} -\Delta u^{(m+1)} +\nabla p^{(m+1)} = n^{(m)} \nabla
 \phi,\\
 \mbox{div }u^{(m+1)}=0.
\end{array}
\right.
\end{equation}
$\bullet$ (Uniform boundedness) \quad If the initial data $(n_0,\,
c_0,\, u_0) \in H^s\times H^{s+1}\times H^s$ with integer $s >
\left[ \frac{d}{2}   \right]+1$, then we show $(n^{(m)}, c^{(m)},
u^{(m)})$ is uniformly bounded in $X_{T_0}^s$ for some $T_0>0$. Let
$\alpha=(\alpha_1, \cdots, \alpha_d)$ be a multi-index and
$|\alpha|:=\alpha_1+\cdots+\alpha_d$. Taking $D^{\alpha}$ operator
on the first equation in \eqref{eq1}, taking scalar product with
$D^{\alpha} n^{(m+1)}$ and summing over $|\alpha| \leq s$, we have
\[
\frac12 \frac{d}{dt} \| n^{(m+1)} \|_{H^{s}}^2 +\| \nabla n^{(m+1)} \|_{H^s}^2
\leq -\sum_{|\alpha|\leq s} \int_{\R^d} D^{\alpha}(u^{(m)}\cdot \nabla n^{(m+1)})\cdot D^{\alpha} n^{(m+1)} dx
\]
\[
+\sum_{|\alpha|\leq s} \int_{\R^d} D^{\alpha}( \chi(c^{(m)}) n^{(m)}
\nabla c^{(m)})\cdot D^{\alpha}\nabla n^{(m+1)} dx
:=I_1+I_2.
\]
Using cancellation and calculus inequality, we obtain
\[
|I_1| \leq \sum_{|\alpha|\leq s}\left|  \int_{\R^d}(D^{\alpha}(u^{(m)}\cdot \nabla n^{(m+1)})- u^{(m)}\cdot\nabla D^{\alpha} n^{(m+1)})\cdot D^{\alpha} n^{(m+1)} dx     \right|
\]
\[
\leq C (\| \nabla u^{(m)}\|_{L^{\infty}} \| n^{(m+1)}\|_{H^s} +\| u^{(m)}\|_{H^s} \| \nabla n^{(m+1)}\|_{L^{\infty}})\| n^{(m+1)}\|_{H^s}
\leq C \|u^{(m)}\|_{H^s} \| n^{(m+1)}\|_{H^s}^2.
\]
Using Young's inequality and interpolation inequality, we have
\[
|I_2|\leq C \sum_{|\alpha|\leq s} \| D^{\alpha} (\chi(c^{(m)}) n^{(m)} \nabla c^{(m)})\|_{L^2}^2 +\frac12 \| \nabla n^{(m+1)} \|_{H^s}^2
\]
\[
\leq C(1+ \| c^{(m)}\|_{H^{s+1}}^{2s}) \| n^{(m)} \|_{H^s}^2 \| c^{(m)} \|_{H^{s+1}}^2+\frac12 \| \nabla n^{(m+1)} \|_{H^s}^2.
\]
We find that
\[
 \frac{d}{dt} \| n^{(m+1)} \|_{H^{s}}^2 +\| \nabla n^{(m+1)} \|_{H^s}^2 \leq C \| u^{(m)} \|_{H^s} \| n^{(m+1)} \|_{H^s}^2
\]
\[
+C(1+ \| c^{(m)}\|_{H^{s+1}}^{2s}) \| n^{(m)} \|_{H^s}^2 \| c^{(m)} \|_{H^{s+1}}^2.
\]
Similarly, we have
\[
 \frac{d}{dt} \| c^{(m+1)}\|_{H^{s+1}}^2 \leq C \| \nabla u^{(m)}\|_{H^s} \| c^{(m+1)} \|_{H^{s+1}}^2
\]
\[
+C(\| \nabla n^{(m)}\|_{H^{s}}+\| n^{(m)}\|_{L^{\infty}})(1+\| c^{(m)}\|_{H^{s+1}}^s) \| c^{(m+1)}\|_{H^{s+1}},
\]
and
\[
\frac{d}{dt} \| u^{(m+1)} \|_{H^s}^2 +\| \nabla u^{(m+1)} \|_{H^s}^2 \leq C \| u^{(m)}\|_{H^s}\| u^{(m+1)} \|_{H^s}^2 +C \| n^{(m)} \|_{H^s}\| u^{(m+1)}\|_{H^s}.
\]
Adding above, we have
\[
 \frac{d}{dt} ( \| n^{(m+1)}\|_{H^s}^2+\| c^{(m+1)} \|_{H^{s+1}}^2 +\| u^{(m+1)}\|_{H^s}^2) +\| \nabla n^{(m+1)}\|_{H^s}^2+\| \nabla u^{(m+1)} \|_{H^s}^2
\]
\[
\leq C( \| u^{(m)}\|_{H^s} +\| \nabla u^{(m)} \|_{H^s} + \| \nabla n^{(m)}\|_{H^s}+\| n^{(m)}\|_{L^{\infty}}+1) \times
\]
\[
( \| n^{(m+1)}\|_{H^s}^2+\| c^{(m+1)} \|_{H^{s+1}}^2 +\| u^{(m+1)}\|_{H^s}^2)+C(1+ \| c^{(m)}\|_{H^{s+1}}^{2s}) \| n^{(m)} \|_{H^s}^2 \| c^{(m)} \|_{H^{s+1}}^2
\]
\[
+C(\| \nabla n^{(m)}\|_{H^{s}}+\| n^{(m)}\|_{L^{\infty}})(1+\| c^{(m)}\|_{H^{s+1}}^s)^2+C\| n^{(m)}\|_{H^s}^2.
\]
Gronwall's inequality gives uniform boundedness in $X_{T_0}^s$ for
some $T_0 >0$, because
\[
\sup (  \| n^{(m+1)}\|_{H^s}^2+\| c^{(m+1)} \|_{H^{s+1}}^2 +\| u^{(m+1)}\|_{H^s}^2)+\int_0^{T_0} \| \nabla n^{(m+1)}\|_{H^s}^2+\| \nabla u^{(m+1)} \|_{H^s}^2 dt
\]
\[
\leq \left( \| n_0\|_{H^s}^2+\| c_0\|_{H^{s+1}}^2+\| u_0\|_{H^s}^2+CT_0^{1/2}(M^{1/2}+M^{s+2})  \right)\exp \left(CT_0^{1/2}(1+M^{1/2})\right),
\]
under the hypothesis
\[
\sup_{t\in [0, T_0]} (  \| n^{(m)}\|_{H^s}^2+\| c^{(m)} \|_{H^{s+1}}^2 +\| u^{(m)}\|_{H^s}^2)+\int_0^{T_0} \| \nabla n^{(m)}\|_{H^s}^2+\| \nabla u^{(m)} \|_{H^s}^2 dt\leq M.
\]
$\bullet$ (Convergence)\quad To show  that $\{ (n^{(m)},\,
c^{(m)},\, u^{(m)})\}$ is a Cauchy sequence in $X^s_T$
for some $0<T_1 <T_0$, we consider the equations of the difference
of solutions
\begin{equation*}
\left\{
\begin{array}{l}
 \partial_t( n^{(m+1)}-n^{(m)})-\Delta (n^{(m+1)}-n^{(m)}) + (u^{(m)} \cdot \nabla ) (n^{(m+1)}-n^{(m)})+(u^{(m)}-u^{(m-1)})\nabla n^{(m)}
 \\ \qquad =-\nabla \cdot [\chi(c^{(m)}) n^{(m)} \nabla c^{(m)}]+\nabla \cdot [\chi(c^{(m-1)}) n^{(m-1)} \nabla c^{(m-1)}],\\
 \partial_t (c^{(m+1)}-c^{(m)}) +(u^{(m)} \cdot \nabla)( c^{(m+1)}-c^{(m)})+(u^{(m)}-u^{(m-1)})\cdot \nabla c^{(m)}\\
 \qquad =-k (c^{(m)}) n^{(m)}+k(c^{(m-1)})n^{(m-1)},\\
 \partial_t (u^{(m+1)}-u^{(m)})-\Delta (u^{(m+1)}-u^{(m)})+(u^{(m)} \cdot \nabla )(u^{(m+1)}-u^{(m)})\\
 \qquad +(u^{(m)}-u^{(m-1)})\cdot\nabla u^{(m)}  +\nabla (p^{(m+1)}-p^{(m)}) = (n^{(m)}-n^{(m-1)}) \nabla
 \phi,\\
\mbox{div }(u^{(m+1)}-u^{(m)})=0.
\end{array}
\right.
\end{equation*}
Following the arguments similarly in \cite{ckl}, we can prove the
convergence. Since its verification is rather straightforward, the
details are omitted.
\end{proof}

To obtain the  blow-up criteria in Theorem \ref{blowup2}, we derive
lengthy a priori estimates. Especially, the estimates of $\| \nabla
u\|_{L^2_tL^{\infty}_x}$ is crucial. To obtain a bound of $\| \nabla
u\|_{L^2_tL^{\infty}_x}$, we first use vorticity estimates to obtain
$L^2$ estimates of $\nabla \omega$, and then, we obtain the
estimates $\| \nabla u \|_{L^{2, \infty}_{t,x}}$ by using the mixed
norms $L^{q,p}_{t,x}$ type estimates for Stokes system (see e.g.
\cite{giso}). Then the desired blow-up criterion can be obtained by
an induction argument. This is the outline of the second part of
Theorem \ref{blowup2} and now we give the proof.
\\
\begin{pfthm3}
Since construction of local solution is done in Proposition
\ref{construct-sol}, it remains to show the blow-up criteria for
$(2D)$ and $(3D)$. We will show those criteria by obtaining a priori
estimates as in the below steps for $[0, T]$ for any $T<T^{*}$,
where $T^*$ is the maximal time of existence. Since $\int_0^{T} \| n
\|_{L^{\infty}}^2 dt <\infty$ and $\| n(t)\|_{L^1} =\| n_0\|_{L^1}$
for all $t\in [0,T]$, we note that $\int_0^T \| n \|_{L^p}^2 dt <
\infty$ for all $p<\infty$.
\\
$\bullet$\,\,(Case $\R^2$)\quad
At first, we consider the case $d=2$.\\
\underline{Step 1-1} ($L^2 \times H^1\times H^1$ Estimates of $(n,
c, u)$).\quad Testing $u$ to the equation of $u$, we have
\[
\frac12 \frac{d}{dt} \| u \|_{L^2}^2 +\| \nabla u \|_{L^2}^2 \leq C\| n \|_{L^2}\| u \|_{L^2}\leq \| n\|_{L^2}^2 +C\| u \|_{L^2}^2.
\]
It follows from integration in time that
\begin{equation}\label{CK10L-march31}
\sup_{0<t\leq T} \| u (t)\|_{L^2}^2 +\int_0^T \| \nabla u \|_{L^2}^2 dt \leq C(\| u_0 \|_{L^2}^2 +\int_0^T \| n \|_{L^2}^2 dt).
\end{equation}
Consider the equation of the vorticity $\omega =\partial_1 u_2-\partial_2 u_1$.
\begin{equation}\label{eq-omega}
\partial_t \omega +(u \cdot \nabla )\omega -\Delta \omega =\nabla \times (n \nabla \phi).
\end{equation}
 Multiplying \eqref{eq-omega} with $\omega$ and integrating, we have
\[
\frac12 \frac{d}{dt} \| \omega \|_{L^2}^2 +\| \nabla \omega \|_{L^2}^2 \leq \left|\int_{\R^2} n \nabla \phi \times \nabla \omega dt \right|
\leq C\| n\|_{L^2} \| \nabla \omega \|_{L^2} \leq C \| n \|_{L^2}^2
+\frac12 \| \nabla \omega \|_{L^2}^2,
\]
and, therefore, we obtain
\[
\sup_{0 <t \leq T} \| \omega(t)\|_{L^2}^2 +\int_0^T \| \nabla \omega \|_{L^2}^2 dt \leq \| \omega_0\|_{L^2}^2 +\int_0^T \| n \|_{L^2}^2 dt.
\]
On the other hand, due to mixed norm estimate of Stokes system (see
e.g. \cite{giso}), we note that for any $p \in (1, \infty)$
\begin{equation}\label{gi-so}
\int_0^T \| \Delta u \|_{L^p}^2 dt \leq C\| u_0 \|_{H^2}^2
+C\int_0^T \| n (t)\|_{L^p}^2 dt +C\int_0^T \| u\|_{L^{2p}}^2\|
\nabla u \|_{L^{2p}}^2 dt< \infty,
\end{equation}
where we used
\[
\int_0^T \| (u\cdot \nabla )u \|_{L^p}^2 dt \leq \int_0^T \| u\|_{L^{2p}}^2\| \nabla u \|_{L^{2p}}^2 dt\leq \| u \|_{L^{\infty}(0, T; L^{2p})}^2 \| \nabla u \|_{L^2(0, T; L^{2p})}^2.
\]
Hence it follows that $\| \nabla u \|_{L^2(0, T; L^{\infty})}<
\infty$.
Next, testing $n$ to the equation of $n$, we obtain
\begin{equation}\label{C10KL-march30}
\frac12 \frac{d}{dt} \| n (t) \|_{L^2}^2 +\| \nabla n \|_{L^2}^2 \leq \int_{\R^2} \chi(c) n \nabla c \nabla n dx
\leq \frac14 \| \nabla n \|_{L^2}^2 +C\| n \|_{L^{\infty}}^2 \| \nabla c \|_{L^2}^2.
\end{equation}
Taking $\nabla $ on the equation of $c$, multiplying $\nabla c$ and integrating over $\R^2$ yield that
\[
\frac12 \frac{d}{dt} \| \nabla c \|_{L^2}^2\leq \left| \int_{\R^2} \nabla u \nabla c \nabla c dx \right| +\left|  \int_{\R^2} \nabla (k(c)n ) \nabla c dx \right|
\]
\begin{equation}\label{C20KL-march30}
\leq C\| \nabla u\|_{L^{\infty}} \| \nabla c\|_{L^2}^2
+C\| n \|_{L^{\infty}}\| \nabla c\|_{L^2}^2 +C\| n \|_{L^{\infty}}\| \nabla c \|_{L^2}^2 +\frac14 \| \nabla n \|_{L^2}^2.
\end{equation}
If we add the above two inequalities \eqref{C10KL-march30} and
\eqref{C20KL-march30}, then we have
\begin{equation}\label{C30KL-march30}
\frac{d}{dt} (\| n(t)\|_{L^2}^2+\| \nabla c(t)\|_{L^2}^2) +\| \nabla n \|_{L^2}^2
\leq C(\| \nabla u \|_{L^{\infty}}+C\| n \|_{L^{\infty}}^2 +1) (\| n \|_{L^2}^2+\| \nabla c \|_{L^2}^2).
\end{equation}
Using Gronwall's Lemma, we have
\[
\sup_{ t \in (0, T]}( \| n (t)\|_{L^2}^2+\| \nabla c(t) \|_{L^2}^2) +\int_0^T \| \nabla n(t)\|_{L^2}^2 dt\leq C(\| n_0 \|_{L^2}^2 +\| \nabla c_0 \|_{L^2}^2) \]
\[
\times\exp (C(T^{1/2}\| \nabla u \|_{L^2(0,T; L^{\infty})}+\| n\|_{L^2(0, T; L^{\infty})}^2 +T))\int_0^T ( \| \nabla u \|_{L^{\infty}}+C\| n \|_{L^{\infty}}^2 +1)dt< \infty.
\]
\underline{Step 1-2}\,\, (Induction argument)\quad Assuming that for
an integer $m$ with $1 \leq m$
\begin{equation}\label{CKL10-march30}
n \in L^{\infty}(0, T; H^{m-1}) \cap L^2(0, T; H^m),\quad c \in
L^{\infty}(0, T: H^{m})
\end{equation}
and
\begin{equation}\label{CKL20-march30}
u \in L^{\infty}(0, T; H^{m})\cap L^2(0, T; H^{m+1}).
\end{equation}
we will show that
\[
n \in L^{\infty}(0, T; H^{m}) \cap L^2(0, T; H^{m+1}),\qquad c \in
L^{\infty}(0, T: H^{m+1}),
\]
\[
u \in L^{\infty}(0, T; H^{m+1})\cap L^2(0, T; H^{m+2}).
\]
First, we take $D^{\alpha}$ operator ($\alpha=(\alpha_1, \alpha_2)$
is a multi index satisfying $|\alpha|=\alpha_1+\alpha_2$,
$|\alpha|\leq m+1$) with the equations of $u$, scalar product them
with $D^{\alpha}u$ and sum over $|\alpha| \leq m+1$, we obtain
\[
\frac12 \frac{d}{dt} \| u \|_{H^{m+1}}^2 +\| \nabla u \|_{H^{m+1}}^2
\leq -\sum_{|\alpha|\leq m+1}\int_{\R^2} D^{\alpha}((u\cdot\nabla)u)
D^{\alpha} u dx + C\| n \|_{H^m} \| \nabla u \|_{H^{m+1}}.
\]
If we use the commutator estimates such that
\[
\abs{\sum_{|\alpha|=0}^{ m+1}\int D^{\alpha}((u\cdot\nabla)u)
D^{\alpha} u } =\abs{\sum_{|\alpha|=0}^{ m+1}\int
[D^{\alpha}((u\cdot\nabla)u)-(u\cdot\nabla)D^{\alpha}u] D^{\alpha} u
}
\leq C\| \nabla u\|_{L^{\infty}} \| u\|_{H^{m+1}}^2,
\]
then we have
\[
 \frac{d}{dt} \| u \|_{H^{m+1}}^2 +\| \nabla u \|_{H^{m+1}}^2 \leq C\| \nabla u \|_{L^{\infty}} \| u \|_{H^{m+1}}^2+C \| n\|_{H^m}^2.
\]
Gronwall's inequality gives us that
\[
u \in L^{\infty}(0, T; H^{m+1})\cap L^2(0, T; H^{m+2}).
\]
Next, we take $D^{\alpha}$ operator ($\alpha=(\alpha_1, \alpha_2)$
is a multi index satisfying $|\alpha|=\alpha_1+\alpha_2$,
$|\alpha|\leq m$) with the equations of $n$ , scalar product them
with $D^{\alpha}n$ and sum over $|\alpha| \leq m$, we obtain
\[
\frac12 \frac{d}{dt} \| n \|_{H^{m}}^2 +\| \nabla n \|_{H^{m}}^2
\leq -\sum_{|\alpha|\leq m}\int_{\R^2} D^{\alpha}((u\cdot\nabla)n)
D^{\alpha} n dx + C\| \chi(c) n \nabla c \|_{H^m} \| \nabla n
\|_{H^{m}}.
\]
Using integration by parts (choose $\alpha_j \ne 0$) and calculus
inequality, we have
\[
\left|\sum_{|\alpha|\leq m}\int_{\R^2}  D^{\alpha}((u\cdot\nabla)n)
D^{\alpha}n dx\right|=\left|\sum_{|\alpha|\leq m}\int_{\R^3}
D^{\alpha-e_j}((u\cdot\nabla)n) D^{\alpha+e_j} n dx\right|
\]
\[
\leq C \| u \cdot \nabla n \|_{H^{m-1}} \| \nabla n \|_{H^m} \leq
C(|u|_{L^{\infty}} \| n \|_{H^m} +C \| u\|_{W^{m-1, \infty}}\|
\nabla n \|_{L^2}) \| \nabla n \|_{H^m}
\]
\bq\label{calcul} \leq C(\| u\|_{L^{\infty}}^2+ \| u\|_{W^{m-1,
\infty}}^2)\| n \|_{H^m}^2 +\frac16 \| \nabla n \|_{H^m}^2. \eq Also
we have
\[
\| \chi(c) n \nabla c\|_{H^m} \leq C \| \nabla c\|_{L^4} \| n
\|_{W^{m,4}} +C \| n \|_{L^{\infty}} \| \chi (c) \nabla c\|_{H^m}
\]
\bq\label{bound-chi} \leq C\| \nabla c \|_{L^2}^{\frac12} \| \nabla
c\|_{H^1}^{\frac12} \| n \|_{H^m}^{\frac12} \| \nabla n
\|_{H^m}^{\frac12} +C_1 \| n \|_{L^{\infty}}(\| \nabla
c\|_{H^m}+C_2), \eq where $C_1$ and $C_2$ are absolute constants
depending only on $\| c\|_{H^m}$, which is bounded in the inductive
assumption $(m-1)$-th step. Using \eqref{calcul} and
\eqref{bound-chi}, we have
\[
\frac12\frac{d}{dt} \| n \|_{H^m}^2 +\| \nabla n \|_{H^m}^2 \leq
C(\| u\|_{L^{\infty}}^2+ \| u\|_{W^{m-1, \infty}}^2)\| n \|_{H^m}^2
\]
\begin{equation}\label{blow-up-est-n}
+C (\| \nabla c\|_{L^2}^2 \|n \|_{H^m}^2+\| n \|_{L^{\infty}}^2)\|
\nabla c\|_{H^m} +C\| n \|_{L^{\infty}}^2 +\frac13 \| \nabla n
\|_{H^m}^2.
\end{equation}
Similarly, taking $H^{m+1}$ scalar product equation of $c$ with
$D^{\alpha} c$ and summing over $|\alpha|\leq m+1$,
\[
\frac12\frac{d}{dt} \| c\|_{H^{m+1}}^2 \leq -\sum_{|\alpha|\leq
m+1}\int_{\R^2} D^{\alpha}((u\cdot\nabla)c) D^{\alpha}c dx + C\|
k(c) n  \|_{H^{m+1}} \| c \|_{H^{m+1}}.
\]
Using commutator estimates
\[
\left|\sum_{|\alpha|\leq m+1}\int_{\R^2}
D^{\alpha}((u\cdot\nabla)c) D^{\alpha} c
dx\right|=\left|\sum_{|\alpha|\leq m+1}\int_{\R^2}
[D^{\alpha}((u\cdot\nabla)c)-(u\cdot\nabla)D^{\alpha}c] D^{\alpha} c
dx\right|
\]
\[
\leq C\| \nabla u\|_{L^{\infty}} \| c\|_{H^{m+1}}^2+C \| \nabla
c\|_{L^4} \| u \|_{W^{m+1, 4}}\| c\|_{H^{m+1}}
\]
\begin{equation}\label{C10KL10-march30}
\leq C\| \nabla u\|_{L^{\infty}} \| c\|_{H^{m+1}}^2 +C\| \nabla
c\|_{L^2}^{\frac12} \| \nabla c\|_{H^1}^{\frac12}\|
u\|_{H^{m+1}}^{\frac12} \|\nabla u\|_{H^{m+1}}^{\frac12} \|
c\|_{H^{m+1}}
\end{equation}
and Leipniz formula
\[
\| k(c) n \|_{H^{m+1}} \leq C \| k (c)\|_{H^{m+1}}\| n
\|_{L^{\infty}} +C \| k(c)\|_{L^{\infty}} \| \nabla n \|_{H^{m+1}}
\]
\[
\leq (C_1\| c\|_{H^{m+1}}+C_2)\| n \|_{L^{\infty}}+C\| \nabla n
\|_{H^{m+1}},
\]
where  $C_1$ and $C_2$ are absolute constants  depending only on $\|
c\|_{H^m}$ bounded in the inductive assumption $(m-1)$-th step, we
have
\[
\frac12\frac{d}{dt} \| c\|_{H^{m+1}}^2 \leq C(\| \nabla u
\|_{L^{\infty}}+ \| n\|_{L^{\infty}}+1) \| c\|_{H^{m+1}}^2+C\| n
\|_{L^{\infty}}^2
\]
\bq\label{blow-up-est-c} +C \| \nabla c\|_{L^2}^2\|
u\|_{H^{m+1}}^2\| \nabla u\|_{H^{m+1}}^2+\frac16 \| \nabla n
\|_{H^{m+1}}. \eq Adding \eqref{blow-up-est-n} and
\eqref{blow-up-est-c}, we have
\[
\frac{d}{dt} (\| n \|_{H^m}^2 +\| c\|_{H^{m+1}}^2 )+\| \nabla n
\|_{H^{m+1}}^2
\]
\[
\leq C( \| \nabla c\|_{L^2}^2 \|n \|_{H^m}^2+\| n
\|_{L^{\infty}}^2+\| \nabla u \|_{L^{\infty}}+ \|
n\|_{L^{\infty}}+1)(\| n \|_{H^m}^2 +\| c\|_{H^{m+1}}^2 )
\]
\[
+C\| n \|_{L^{\infty}}^2+C \| \nabla c\|_{L^2}^2\| u\|_{H^{m+1}}^2\|
\nabla u\|_{H^{m+1}}^2.
\]
Since
\[
\| \nabla c\|_{L^{\infty}(0, T; L^2)}^2 \| n \|_{L^2(0, T; H^m)}^2
+\| n \|_{L^2(0, T; L^{\infty})}^2\]
\[+\| \nabla u \|_{L^1(0, T; L^{\infty})}
+\| \nabla c\|_{L^{\infty}(0, T; L^2)}^2\| u \|_{L^{\infty}(0, T;
H^{m+1})}^2 \| \nabla u \|_{L^2(0, T; H^{m+1})}^2<\infty,
\]
it follows via Gronwall's inequality that
\[
\| n \|_{L^{\infty}(0, T; H^m)} +\| c\|_{L^{\infty}(0, T; H^{m+1})}
+\| \nabla n \|_{L^2(0, T; H^m)} < \infty.
\]
This completes the proof of 2D case.
\\
%
\\
$\bullet$\,\,(Case $\R^3$)\quad Next, we consider the case $d=3$.\\
\underline{Step 2-1}\,\,($L^2 \times H^1\times H^1$ Estimates of
$(n, c, u)$).\quad Following similar computations as in 2D case, we
also have the estimate \eqref{CK10L-march31}.
We recall the equation of the vorticity $\omega =\nabla \times
\omega$
\begin{equation}\label{eq-omega-3d}
\partial_t \omega +(u \cdot \nabla )\omega -\Delta \omega =(\omega \cdot \nabla)u+\nabla \times (n \nabla \phi).
\end{equation}
Multiplying \eqref{eq-omega} with $\omega$ and integrating in
spatial variables, we have
\[
\frac12 \frac{d}{dt} \| \omega \|_{L^2}^2 +\| \nabla \omega \|_{L^2}^2 \leq \left|\int_{\R^2} n \nabla \phi \times \nabla \omega dt \right|+\left| \int_{\R^2} | u\omega| |\nabla \omega| dx   \right|
\]
\[
\leq C\| n\|_{L^2} \| \nabla \omega \|_{L^2} +\| u \omega \|_{L^2}\| \nabla \omega\|_{L^2}\leq C \| n \|_{L^2}^2 +C \| u \omega \|_{L^2}^2+\frac14 \| \nabla \omega \|_{L^2}^2
\]
\[
\leq C \| u \|_{L^\beta}^2 \| \omega \|_{L^{\frac{2\beta}{\beta-2}}} +C \| n \|_{L^2}^2+\frac14 \| \nabla \omega \|_{L^2}^2
\leq C\| u \|_{L^\beta}^{\frac{2\beta}{2\beta-3}} \| \omega
\|_{L^2}^2 +C \| n \|_{L^2}^2+\frac12 \| \nabla \omega \|_{L^2}^2
\]
and, therefore, we have
\[
\sup_{0 <t \leq T} \| \omega(t)\|_{L^2}^2 +\int_0^T \| \nabla \omega \|_{L^2}^2 dt
\leq \left(\| \omega_0\|_{L^2}^2 +C\int_0^T \| n \|_{L^2}^2 dt\right)\exp
\left(C\int_0^T \| u \|_{L^\beta}^{\frac{2\beta}{2\beta-3}} dt    \right).
\]
Using the mixed norm estimate of Stokes system, we note that
\begin{equation}\label{gi-so-32}
\int_0^T \| \Delta u \|_{L^3}^2 dt \leq  C\| u_0 \|_{H^1}^2
+C\int_0^T \| n (t)\|_{L^3}^2 dt+C\int_0^T \| u\|_{L^{6}}^2\| \nabla
u \|_{L^{6}}^2 dt< \infty,
\end{equation}
where we used
\[
\int_0^T \| (u \cdot\nabla )u \|_{L^3}^2 dt \leq \int_0^T \| u\|_{L^{6}}^2\| \nabla u \|_{L^{6}}^2 dt\leq \| \omega \|_{L^{\infty}(0, T; L^{2})}^2 \| \nabla \omega\|_{L^2(0, T; L^{2})}^2.
\]
Again with aid of the estimate of Stokes system,
we have
\begin{equation}\label{gi-so-42}
\int_0^T \| \Delta u \|_{L^4}^2 dt \leq C\| u_0 \|_{H^2}^2
+C\int_0^T \| n (t)\|_{L^4}^2 dt+C\int_0^T \| u\|_{L^{6}}^2\| \nabla
u \|_{L^{12}}^2 dt< \infty,
\end{equation}
where we used
\[
\int_0^T \| (u\cdot \nabla )u \|_{L^4}^2 dt \leq \int_0^T \| u\|_{L^{6}}^2\| \nabla u \|_{L^{12}}^2 dt\leq \| \omega \|_{L^{\infty}(0, T; L^{2})}^2 \| \Delta u\|_{L^2(0, T; L^{3})}^2.
\]
Hence it is direct that $\| \nabla u \|_{L^2(0, T; L^{\infty})}<
\infty$. Next, testing $n$ to the equation of $n$ as in 2D case, we
also have \eqref{C10KL-march30}. For equation of $c$, we can obtain
\eqref{C20KL-march30} without any modification, and therefore, it is
immediate that \eqref{C30KL-march30}.
Using Gronwall's Lemma, we obtain
\[
\sup_{ t \in (0, T]}( \| n (t)\|_{L^2}^2+\| \nabla c(t) \|_{L^2}^2) +\int_0^T \| \nabla n(t)\|_{L^2}^2 dt\leq C(\| n_0 \|_{L^2}^2 +\| \nabla c_0 \|_{L^2}^2) \]
\[
\times\exp (C(T^{1/2}\| \nabla u \|_{L^2(0,T; L^{\infty})}+\| n\|_{L^2(0, T; L^{\infty})}^2 +T))\int_0^T ( \| \nabla u \|_{L^{\infty}}+C\| n \|_{L^{\infty}}^2 +1)dt< \infty.
\]
\underline{Step 2-2}\,\, (Induction argument)\quad As in 2D case,
most of all estimates are the same as those given above. Therefore,
we just mention different estimates compared to 2D case. Up to
estimate \eqref{blow-up-est-n}, all estimates are exactly the same
as before and however, the following is slightly different form of
estimate (compare to \eqref{C10KL10-march30}). Indeed, using
commutator estimates,
\[
\left|\sum_{|\alpha|\leq m+1}\int_{\R^3} D^{\alpha}((u\cdot\nabla)c)
D^{\alpha} c dx\right|=\left|\sum_{|\alpha|\leq m+1}\int_{\R^3}
[D^{\alpha}((u\cdot\nabla)c)-(u\cdot\nabla)D^{\alpha}c] D^{\alpha} c
dx\right|
\]
\[
\leq C\| \nabla u\|_{L^{\infty}} \| c\|_{H^{m+1}}^2+C \| \nabla
c\|_{L^3} \| u \|_{W^{m+1, 6}}\| c\|_{H^{m+1}},
\]
\[
\leq C\| \nabla u\|_{L^{\infty}} \| c\|_{H^{m+1}}^2+C\| \nabla
c\|_{L^2}^{\frac12} \| \nabla c\|_{H^1}^{\frac12} \|\nabla
u\|_{H^{m+1}} \| c\|_{H^{m+1}}.
\]
With the above modification, to sum up, we have
\[
\frac{d}{dt} (\| n \|_{H^m}^2 +\| c\|_{H^{m+1}}^2 )+\| \nabla n
\|_{H^{m+1}}^2
\]
\[
\leq C( \| \nabla c\|_{L^2}^2 \|n \|_{H^m}^2+\| n
\|_{L^{\infty}}^2+\| \nabla u \|_{L^{\infty}}+ \|
n\|_{L^{\infty}}+\| \nabla u \|_{H^{m+1}}+1)(\| n \|_{H^m}^2 +\|
c\|_{H^{m+1}}^2 ) +C\| n \|_{L^{\infty}}^2.
\]
Under the same assumption as \eqref{CKL10-march30} and
\eqref{CKL20-march30}, Gronwall's inequality implies that
\[
\| n \|_{L^{\infty}(0, T; H^m)} +\| c\|_{L^{\infty}(0, T; H^{m+1})}
+\| \nabla n \|_{L^2(0, T; H^m)} < \infty.
\]
This finishes the case of 3D and therefore, proof is completed.
\end{pfthm3}

\section{{Proof of Theorem \ref{Theorem6}}}

%

In this section, we present the proof of Theorem \ref{Theorem6}. The
following lemma shows weighted energy estimate and truncated energy
estimate shown in \cite{CPZ2} in case that fluid is not coupled. It
is remarkable that even in the presence of fluid equations,
influence of fluid does not appear. Indeed, incompressibility causes
cancelation of terms involving velocity of fluid, which is a crucial
observation for the proof of Theorem \ref{Theorem6}.

\begin{lemma}\label{weighted-estimate}
Let $\phi(\cdot)$ be an auxiliary function such that $\phi'(c) -
\phi(c)\chi(c)=0$ and $K$ a positive number. Then, the classical
solutions to \eqref{adKS} satisfy the following weighted energy
equality \eqref{auxillary} and truncated energy equality
\eqref{truncated}:
\[
\frac d {dt} \int_{\R^d} \left( \frac n {\phi(c)} \right)^p \phi(c)
+ 4  \frac {p-1} p \int_{\R^d} \phi(c) \left| \na \left( \frac n
{\phi(c)} \right)^{p/2} \right|^2
\]
\begin{equation}\label{auxillary}
= (p-1) \int_{\R^d} \phi^2 (c) \chi(c) k (c) \left( \frac n
{\phi(c)} \right)^{p+1},
\end{equation}
\[
\ddt  \int_{\R^d} \left( \frac n {\phi(c)} - K \right)_+^p \phi(c)
+2\frac{p-1}{p} \int \phi(c) \abs{\na \left( \frac n {\phi(c)} - K
\right)_+^{p/2}}^2
\]
\[
=(p-1)\int \left( \frac n {\phi(c)} - K \right)_+^{p+1} \phi^2(c)
\chi(c) k(c) + (2p-1)K \int \phi^2(c) \chi(c)k(c) \left( \frac n
{\phi(c)} - K
 \right)_+^p
\]
\begin{equation}\label{truncated}
 + pK^2 \int \phi(c)^2 \chi(c)k(c) \left( \frac n {\phi(c)} - K
\right)_+^{p-1}.
\end{equation}


 \end{lemma}
\begin{proof}
We note first that
\begin{align*}
\ddt \frac{n}{\phi} &= \frac{n_t \phi -n\phi' c_t}{\phi^2} = \frac
1\phi (\Del n - \na\cdot(n\chi\na c) - u\cdot \na n) - \frac {n\phi'
c_t}{\phi^2}.
\end{align*}
Due to $\na \cdot u = 0$ and $\phi' = \phi \chi$, we also observe
that
\[
\Del n - \na\cdot(n\chi\na c) - u\cdot \na n =  \na \cdot\left( \phi
\na \left (\frac{n}{\phi} \right) -  \frac{n}{\phi} \phi u \right)
\]
and therefore, it follows that
\begin{equation}\label{CK10L10-march30}
\ddt \frac{n}{\phi}  = \frac{1}{\phi} \bkt{\na \cdot \left ( \phi
\na \left(\frac{n}{\phi} \right)- \frac{n}{\phi} \phi u \right)
 - n\chi c_t }.
\end{equation}
Testing $p(\frac{n}{\phi})^{p-1}\phi$ to \eqref{CK10L10-march30} and
using the integration by parts, we obtain
\[
\ddt \int \psca^p \phi= \int p \psca^{p-1} \bkt{ \na \cdot \left (
\phi \na \left(\frac{n}{\phi} \right)- \frac{n}{\phi} \phi u \right)
 - n\chi c_t} +  \psca^p \phi\chi c_t dx
\]
\begin{equation}\label{CK20L20-march30}
=\int p \psca^{p-1}  \na \cdot \left ( \phi \na \left(\frac{n}{\phi}
\right)\right)dx- \int p \psca^{p-1}  \na \cdot \left
(\frac{n}{\phi} \phi u \right)dx -(p-1)\int\psca^p \phi \chi c_t dx.
\end{equation}
%
We estimate separately each term in \eqref{CK20L20-march30}.
\begin{align}\label{1}
\int p \psca^{p-1}  \na \cdot \left ( \phi \na \left(\frac{n}{\phi}
\right) \right) = -p(p-1)\int \phi \psca^{p-2} \abs{\na \psca}^2 dx,
\end{align}
\begin{align}\label{2}\begin{aligned}
&p\int \psca^{p-1} \na \cdot \left (  \sca \phi u \right) dx = p\int
\psca^{p-1} \psca \na \cdot (\phi u)
+  \psca ^{p-1} \na \psca \cdot \phi u dx \\
&\quad= p\int \psca^p \na \cdot (\phi u) dx + \int \na \psca^p \cdot
\phi u dx = (p-1) \int \psca^p \na \cdot (\phi u) dx,
\end{aligned} \end{align}
\begin{align}\label{3}\begin{aligned}
&(p-1)\int \psca^p \phi \chi c_t dx = -(p-1) \int \psca^p \phi (\chi u\cdot \na c + k n) dx \\
&\qquad = -(p-1)\int \psca^p \phi' u \cdot \na c dx  - (p-1)\int \psca^p \phi\chi k n dx \\
&\qquad =-(p-1)\int \psca^p \na \cdot(u\phi) dx - (p-1)\int \phi^2
\psca ^{p+1} \chi k dx.
\end{aligned}\end{align}
Adding up \eqref{1}-\eqref{3}, we obtain \eqref{auxillary}. By the
Sobolev inequality, for any $p$ with $ \max\{ 1, d/2 -1\} \le p
<\infty$, it follows that
\begin{equation}\label{sobolev}
\frac d {dt} \int_{\R^d} ( \frac n {\phi(c)} )^p \phi(c)\le (p-1)
\left\Vert \na ( \frac n {\phi(c)} )^{\frac{p}{2}} \right\Vert_{L^2
(\R^d)}^2
 \cdot \left[  \tilde C(d) K_1
 \left\Vert \phi^{2/d}(c) ( \frac n {\phi(c)} ) \right\Vert_{L^{d/2}(\R^d)}
- \frac 4 p \right],
\end{equation}
as long as $0 \le c\le \| c_0\|_{L^{\infty}}$ and
$\displaystyle\sup_{0\le c \le \| c_0\|_{L^{\infty}} }
\phi^2(c)\chi(c)k(c): = K_1 <\infty$.
%
For the truncated inequality, we proceed similar computations as in
those of the weighted equality. We obtain
\[
\ddt \int \tsca^p \phi  = \int p \tsca^{p-1} \ddt \psca \phi
+\tsca^p \phi' c_t
\]
\[
= \int p \tsca^{p-1}  \na \cdot \left ( \phi \na
\left(\frac{n}{\phi} \right) - \sca \phi u \right)- p\int
\tsca^{p-1} n\chi  c_t  + \int \tsca^p \phi\chi c_t.
\]
We note that integration by parts yields
\begin{equation}\label{CK30L30-march30}
\int p \tsca^{p-1}   \na \cdot \left ( \phi \na \left(\frac{n}{\phi}
\right) \right) dx=
 -p(p-1)\int \phi \tsca^{p-2} \left|\na \psca\right|^2 dx.
\end{equation}
With the aid of replacement of $\sca$ by $\left(\sca - K \right) +K$
and integration by parts,
%
%
it is direct that
\[
-p \int \tsca^{p-1} \na \cdot \left (\sca \phi u \right)  = -pK \int
\tsca^{p-1} \na \cdot (\phi u)
\]
\begin{equation}\label{CK40L40-march30}
-(p-1)\int \tsca^p \na\cdot (\phi u).
\end{equation}
It is also straightforward that
\[
p\int \tsca^{p-1} n\chi \na \cdot (uc)- \int \tsca^p \phi\chi \na
\cdot(uc) dx
\]
\[
= p\int \tsca^{p-1} \left(\sca -K + K  \right)\phi \chi \na \cdot
(uc)- \int \tsca^p \phi\chi \na \cdot(uc) dx
\]
\begin{equation}\label{CK50L50-march30}
= (p-1) \int \tsca^p \phi \chi \na \cdot (uc) dx  + p K \int
\tsca^{p-1} \phi \chi \na \cdot(uc) dx
\end{equation}
As noticed earlier, due to $\phi \chi \na \cdot (uc)=u\phi' \na c =
\na \cdot ( \phi u)$, \eqref{CK40L40-march30} and
\eqref{CK50L50-march30} are cancelled out each other. We also
observe that
\[
p\int \tsca^{p-1} n^2\chi k=p\int \tsca^{p-1} \left ( \sca -K
+K\right)^2 \chi k\phi^2
\]
\begin{equation}\label{CK60L60-march30}
= p\int \tsca^{p+1} \phi^2 \chi k + 2pK\int \tsca^p \phi^2 \chi k  +
pK^2 \int \tsca^{p-1} \phi^2 \chi k,
\end{equation}
\begin{equation}\label{CK70L70-march30}
-\int \tsca^p \phi \chi k n= -\int \tsca^{p+1} \phi^2 \chi k - K
\int \tsca^p \phi^2 \chi k.
\end{equation}
Summing up \eqref{CK30L30-march30}-\eqref{CK70L70-march30}, we
obtain \eqref{truncated}. This completes the proof.
\end{proof}

With the help of Lemma \ref{weighted-estimate}, the remaining
procedures of the proof of Theorem \ref{Theorem6} are almost
identical with those in \cite{PeVa} and, however, we give the sketch
of the proof for clarity.\\

\begin{pfthm4}
Let $\nu(t)$, $\eta(t)$ be auxiliary functions, which will be
specified later and $\phi(c)$ given in Lemma
\ref{weighted-estimate}. We define a truncated energy $E(\xi)$ by
\[
E(\xi): = \sup_{0\le t\le T} \int_{\R^d} \left( \frac n {\phi(c)} - \xi \eta(t) \right)_+^p
+ 2\frac{p-1}{p} 
 \int_0^T\int  \abs{ \na \left( \frac n {\phi(c)} - \xi \eta(t)
 \right)_+^{p/2}}^2.
\]
Using \eqref{truncated} and following
similar procedures as in the proof of Theorem \ref{Theorem4}, it
follows  that under the condition $\eta(0) >\| n_0\|_{L^{\infty}}$
\begin{align}\label{E(xi)} \begin{aligned}
\phi_{min} E(\xi)
 &
 \le \sup_{0 \le t \le T} \int_{\R^d} \left( \frac n {\phi(c)} - \xi \eta(t) \right)_+^p
 \phi(c)\\
 &\qquad + 2 \frac {p-1} p \int_0^T \int_{\R^d} \phi(c)
 \left| \na \left( \frac n {\phi(c)} - \xi \eta(t) \right)_+^{p/2}
 \right|^2
\\
&\le - \xi \int_0^T \dot \eta (t) \int_{\R^d} \left( \frac n
{\phi(c)} - \xi \eta(t) \right)_+^{p-1} \phi(c)
\\
&\quad + (2p-1)\xi \int_0^T  \eta(t) \int_{\R^d} \phi^2 (c) \chi(c)
\kappa(c)
 \left( \frac n {\phi(c)} - \xi \eta(t) \right)_+^p
\\
& \quad + p\xi^2 \int_0^T  \eta(t)^2 \int_{\R^d} \phi^2 (c) \kappa(c)
 \left( \frac n {\phi(c)} - \xi \eta(t) \right)_+^{p-1}.
\end{aligned} \end{align}
By the sobolev embedding it holds that
\begin{align}\label{qembed}
\left\Vert \left( \frac n {\phi(c)} - \xi \eta(t) \right)_+\right\Vert_{L^q([0,T]\times \bbr^d)}^p
\le C E(\xi), \quad q = p(d+2)/d.
\end{align}
On the other hands, we define the level set energy
\[ U(\xi) : = \int_0^T \int \nu(t)\left ( \frac n {\phi(c)} - \xi \eta(t) \right)_+^p dxdt.\]
Differentiating in $\xi$,
\begin{align*}
 U'(\xi)& = -\int_0^T\int p\nu(t)\eta(t) \left( \frac n {\phi(c)} - \xi \eta(t) \right)_+^{p-1} dxdt.
 \end{align*}
 Interpolating $ p-1 < p <q $, we have
\begin{align*}
 U(\xi)& \le  \left( \int_0^T \int \nu(t)^{\frac{p-1}{p(1-\theta)}}
\left( \frac n {\phi(c)} - \xi \eta(t) \right)_+^{p-1} dxdt \right)^{\al}
 E(\xi)^{\theta},
\end{align*}
for $\theta = \frac{d+2}{d+2p},$ $ \al = \frac{2p}{d+2p}$.
  Let the auxiliary functions $\nu(\xi), \eta(\xi)$ satisfy that
 \[\nu(t)^{\frac{p-1}{p(1-\theta)}} + |\dot{\eta}(t)| \le C_1\nu(t) \eta(t),
 \quad  \eta(c) \le C_2\nu(t).\] Then $E(\xi)$, $U(\xi)$ satisfy the following
 differential inequalities:
 \begin{align}\label{ode}\begin{aligned}
 \begin{cases}
U(\xi) \le C_1^{\al}|U'(\xi)|^{\al}E(\xi)^{\theta}, \\
 E(\xi) \le \xi |U'(\xi)|(C_1+ C_2\xi) + C_2\xi U(\xi),
 \end{cases}
\end{aligned} \end{align}
Firstly we choose $\nu(t) = (1+ t)^{1-\frac{d}{2p}}$ and $\eta(t) =
(1+t)^{-1}$ with $C_1= (1+T)^{1-\frac{d}{2p}}$.
 Working \eqref{ode} with $G(\xi) = U(\xi)^{a}$ for some $0< a<1$,
 we arrive at $G(\xi)$ vanishing for a finite $\xi_1$(See (step 4) for Theorem $4.1$ in \cite{PeVa})
 under the condition $ p> \frac{d+2}{2}$. The same holds for $U(\xi)$ and we obtain the decay
 \[ \|n(t)\|_{L^{\infty}} \le C(T)\nu(t).\]
 Next, we choose $\nu(t)= \eta(t)= (1+t)^{-1}$
 with $C_2 = (1+T)^{1 - \frac{d}{2p}}$ to relax the initial integrability of $n_0$
 to $p > \frac{d(d+2)}{2(d+2)}$. Up to this point, $L^{\infty}$ decay of $n(t)$ depends on $T$ for
 $0<t<T$.
When $ \|n_0\|_{L^{\frac d2}}$ is small enough,
 \eqref{auxillary} gives
\[  \int_{\R^d} \left( \frac n {\phi(c)} \right)^{\frac d2} \phi(c)
  + \int_{\R^d} \phi(c) \left| \na \left( \frac n {\phi(c)} \right)^{d/4} \right|^2
 \le C  \|n_0\|_{L^{\frac d2}},\]
from which we have $\| \frac{n}{\phi(c)}\|_{L^{\frac
{d+2}{2}}_{t,x}([0,1]\times \bbr^d)} < C$,
 and  $\| n(t_0 )\|_{L^{\frac{d+2}{2}}} \le C$ for $ t_0 \le 1$. Now using the result for
$p > \frac{d(d+2)}{2(d+2)}$ and scale invariance of the norm $\|
n_0\|_{L^{\frac d2}}$ and $ \|c_0\|_{L^{\infty}}$ under scaling
\eqref{C10KL-march15}, we conclude $\|n(t)\|_{L^{\infty}} \le
\frac{C}{t}$ for a uniform constant $C$. This completes the proof.
\end{pfthm4}

%

%
%

\section*{Acknowledgements}
M. Chae's work was partially supported by NRF-2011-0028951. K.
Kang's work was partially supported by NRF-2012R1A1A2001373.  J.
Lee's work was partially supported by NRF-2011-0006697 and Chung-Ang
University Research Grants in 2013.

\begin{equation*}
\left.
\begin{array}{cc}
{\mbox{Myeongju Chae}}\qquad&\qquad {\mbox{Kyungkeun Kang}}\\
{\mbox{Department of Applied Mathematics }}\qquad&\qquad
 {\mbox{Department of Mathematics}} \\
{\mbox{Hankyong National University
}}\qquad&\qquad{\mbox{Yonsei University}}\\
{\mbox{Ansung, Republic of Korea}}\qquad&\qquad{\mbox{Seoul, Republic of Korea}}\\
{\mbox{mchae@hknu.ac.kr }}\qquad&\qquad {\mbox{kkang@yonsei.ac.kr }}
\end{array}\right.
\end{equation*}
\begin{equation*}
\left.
\begin{array}{c}
{\mbox{Jihoon Lee}}\\
{\mbox{Department of Mathematics }}\\
{\mbox{Chung-Ang University}}\\
{\mbox{Seoul, Republic of Korea}}\\
{\mbox{jhleepde@cau.ac.kr }}
\end{array}\right.
\end{equation*}

\end{document}